\documentclass[sn-mathphys,Numbered]{sn-jnl}

\usepackage{graphicx}%
\usepackage{multirow}%
\usepackage{amsmath,amssymb,amsfonts}%
\usepackage{amsthm}%
\usepackage{mathrsfs}%
\usepackage[title]{appendix}%
\usepackage{xcolor}%
\usepackage{textcomp}%
\usepackage{manyfoot}%
\usepackage{booktabs}%
\usepackage{algorithm}%
\usepackage{algorithmicx}%
\usepackage{algpseudocode}%
\usepackage{listings}%
\usepackage{mathtools}
\mathtoolsset{showonlyrefs=true}
\usepackage{subcaption}
\newcommand{\e}[1]{\begin{equation}#1\end{equation}}

\newcommand{\al}[1]{\begin{align}#1\end{align}}
\newcommand{\ald}[1]{\begin{aligned}#1\end{aligned}}
\newcommand{\ite}[1]{\begin{itemize}#1\end{itemize}}

\newcommand{\prn}[1]{\left(#1\right)}

\newcommand{\cur}[1]{\left\{#1\right\}}

\newcommand{\domain}{\mathop{\mathrm{dom}}\nolimits}
\newcommand{\image}{\mathop{\mathrm{Im}}\nolimits}
\newcommand{\kernel }{\mathop{\mathrm{Ker}}\nolimits}

\def\coloneq{\mathrel{\mathop:}=}
\def\r{\mathbb{R}}
\def\rn{\mathbb{R}^n}
\def\rm{\mathbb{R}^m}
\def\rd{\mathbb{R}^d}
\def\sn{\mathbb{S}^n}
\def\sd{\mathbb{S}^d}

\newtheorem{thm}{Theorem}%
\newtheorem{prp}[thm]{Proposition}%
\newtheorem{crl}[thm]{Corollary}%

\newtheorem{rmk}{Remark}%

\newtheorem{dfn}{Definition}%
\newtheorem{asm}{Assumption}

\begin{document}

\title[Article Title]{Variational analysis of unbounded and discontinuous generalized eigenvalue functions with application to topology optimization}


\author*[1]{\fnm{Akatsuki} \sur{Nishioka}}\email{akatsuki\_nishioka@mist.i.u-tokyo.ac.jp}

\author[1,2]{\fnm{Yoshihiro} \sur{Kanno}}

\affil*[1]{\orgdiv{Department of Mathematical Informatics}, \orgname{The University of Tokyo}, \orgaddress{\street{Hongo 7‑3‑1}, \city{Bunkyo‑ku}, \postcode{113-8656}, \state{Tokyo}, \country{Japan}}}

\affil[2]{\orgdiv{Mathematics and Informatics Center}, \orgname{The University of Tokyo}, \orgaddress{\street{Hongo 7‑3‑1}, \city{Bunkyo‑ku}, \postcode{113-8656}, \state{Tokyo}, \country{Japan}}}

\abstract{
The maximum (or minimum) generalized eigenvalue of symmetric positive semidefinite matrices that depend on optimization variables often appears as objective or constraint functions in structural topology optimization when we consider robustness, vibration, and buckling. It can be an unbounded or discontinuous function where matrices become singular (where a topological change of the structural design occurs). Based on variational analysis, we redefine the maximum (and minimum) generalized eigenvalue function as an extended real-valued function and propose a real-valued continuous approximation of it. Then, we show that the proposed approximation epi-converges to the original redefined function, which justifies solving problems with the approximation instead. We consider two specific topology optimization problems: robust compliance optimization and eigenfrequency optimization and conduct simple numerical experiments.
}



\maketitle


\section{Introduction}

Topology optimization is a method to find an optimal design of a structure, which allows topological changes of a structure from the initial design \cite{allaire02,bendsoe04}. Although topology optimization problems of a continuum structure are essentially infinite-dimensional, this paper only deals with discretized problems or truss topology optimization problems, which are inherently finite-dimensional.

When we consider robustness, vibration, and buckling in topology optimization, generalized eigenvalues of symmetric matrices naturally appear as objective or constraint functions \cite{achtziger07siam,ferrari19,kocvara02,ohsaki99,torii17}. Here, the optimization (or design) variable is a vector that consists of cross-sectional areas of bars, densities of each finite element, etc, and we consider generalized eigenvalues of matrix-valued functions of the optimization variable. These generalized eigenvalue functions can be unbounded (extended real-valued) or discontinuous where matrices become singular (where some of the design variables become 0 and a topological change of the design occurs) \cite{achtziger07siam}. One conventional approach to avoid such inconvenience is to set very small positive lower bounds on the design variables (or similarly, fill holes with a very weak material called an ersatz material) \cite{allaire02,bendsoe04,nishioka23coap}. However, when the objective function is discontinuous, this approach is not necessarily an appropriate way to approximate the original problem, because solutions of approximated problems do not necessarily converge to a solution to the original problem.

In this paper, using variational analysis \cite{rockafellar98}, which is useful when we treat extended real-valued or discontinuous functions, we propose a new approximation to the unbounded and discontinuous generalized eigenvalue function. To this end, we first redefine the maximum (and minimum) generalized eigenvalue as an extended real-valued function. Then, we show that the proposed approximation epi-converges to the redefined function, which ensures the convergence of solutions of the approximated optimization problems to solutions of the original optimization problems in some sense. We first consider a general setting and then consider some specific generalized eigenvalues in topology optimization: robust (or worst-case) compliance and the minimum (or fundamental) eigenfrequency. Note that generalized eigenvalue optimization also has application to control theory \cite{boyd93,nesterov95}.

Contributions are summarized as follows:
\ite{
\item Introduce an extended definition of the maximum (and minimum) generalized eigenvalue. The differences from the definition introduced by \cite{achtziger07siam} are as follows.
\ite{
\item We both consider the maximum and the minimum generalized eigenvalue and their relationships.
\item We consider the case $\lambda_\mathrm{max}(X,Y)$ with matrix variables $X,Y\in\sn_{\succeq0}$. The results with matrix-valued functions $\lambda_\mathrm{max}(A(x),B(x))$ can be obtained as direct consequences.
\item We consider the case when both matrices in the arguments are zero matrices.
\item We redefine the maximum (and minimum) generalized eigenvalue as an extended real-valued function, which can take $+\infty$ as a function value.
}
\item Propose a new approximation of the maximum generalized eigenvalue and prove epi-convergence of it to the proposed extended real-valued function. The advantages of the proposed approximation are as follows:
\ite{
\item It is real-valued and continuous, and thus it is easier to deal with in optimization problems. 
\item It epi-converges to the original redefined function, and thus there exists a theoretical guarantee of the convergence of solutions of the approximated optimization problem to solutions of the original optimization problem in some sense, unlike the conventional approach with the artificial lower bound.
\item It does not require artificial lower bounds on variables, and thus the proposed approximated problems for structural optimization are actual topology optimization problems. We actually obtain solutions with many variables equal to zero in numerical experiments.
}
}

We also conduct numerical examples on simple truss topology optimization problems. Note that most of the theoretical results also hold for continuum topology optimization.

\section{Notation and preliminaries}

\subsection{Notation}

We use the following notation: $\mathbb{R}^m_{\ge0}$ and $\mathbb{R}^m_{>0}$ are the set of $m$-dimensional real vectors with nonnegative and positive components, respectively. We denote the zero matrix by $0$. $A\succeq 0$ and $A\succ 0$ respectively denote that $A$ is symmetric positive semidefinite and symmetric positive definite. $\mathbb{S}^n$, $\mathbb{S}^n_{\succeq 0}$, and $\mathbb{S}^n_{\succ 0}$ are the sets of $n$-dimensional symmetric matrices, symmetric positive semidefinite matrices, and symmetric positive definite matrices, respectively. For an $n$-dimensional square matrix $A$, we define $\kernel A\coloneq\{v\in\rn\mid Av=0\}$ and $\image A\coloneq\{Av\mid v\in\rn\}$. $\langle a,b \rangle$ and $\|a\|$ denote the standard inner product and the Euclidean norm for $a,b\in\rm$. $\langle A,B \rangle\coloneq\mathrm{tr}(AB)$ denotes an inner product of matrices $A,B\in\sn$. $\lambda_\mathrm{max}(X,Y)$ and $\lambda_\mathrm{min}(X,Y)$ denote the maximum and the minimum generalized eigenvalues of a pair of matrices $(X,Y)$. Eigenvalue with one argument $\lambda_\mathrm{max}(X)$ denotes the maximum standard eigenvalue. $\domain f$ denotes the effective domain of $f:\rm\to(-\infty,\infty]$, i.e.~$\domain f\coloneq\{x\in\rm\mid f(x)<\infty\}$. 

\subsection{Generalized eigenvalues}

In this paper, we only consider generalized eigenvalues of symmetric matrices, and thus they are real numbers. See \cite{harville97,boyd93} for details.

\subsubsection{Positive definite matrices}

\begin{dfn}[Generalized eigenvalues of symmetric positive definite matrices]
We define generalized eigenvalues and generalized eigenvectors of $(X,Y)\in\sn_{\succ0}\times\sn_{\succ0}$ by $\lambda_i\in\r$ and nonzero vectors $v_i\in\rn$ ($i=1,\ldots,n$) satisfying
\e{
Xv_i=\lambda_i Yv_i.
\label{ge}
}
\end{dfn}
 
Generalized eigenvalue problem \eqref{ge} is equivalent to the standard eigenvalue problem with the matrix $Y^{-1/2}XY^{-1/2}$, where $Y^{-1/2}$ is the positive square root of a symmetric positive definite matrix $Y$ \cite{bhatia13}. This fact ensures that there exist $n$ generalized eigenvalues. Also, the continuity of $\lambda_\mathrm{max}(\cdot,\cdot)$ on $\sn_{\succ0}\times\sn_{\succ0}$ follows by the fact that $\lambda_\mathrm{max}(X,Y)=\lambda_\mathrm{max}(Y^{-1/2}XY^{-1/2})$ (see \cite{nishioka23coap} for details). For any $X,Y\in\sn_{\succ0}$, the maximum generalized eigenvalue can be characterized by using the generalized Rayleigh quotient:
\e{
\lambda_\mathrm{max}(X,Y)=\underset{v\neq0}{\sup}\frac{v^\top Xv}{v^\top Yv}.
\label{rayleigh}
}
Since $\lambda_\mathrm{max}(X,Y)\le\alpha$ is equivalent to $\alpha Y-X\succeq0$ for any $\alpha\in\r$, sublevel sets of $\lambda_\mathrm{max}(X,Y)$ are convex, hence $\lambda_\mathrm{max}(\cdot,\cdot)$ is quasiconvex on $\sn_{\succ0}\times\sn_{\succ0}$ \cite{achtziger07siam}.
\begin{dfn}[Quasiconvex function\footnote{Not to be confused with a different notion of a quasiconvex function (in the sense of Morrey \cite{morrey52}) often used in the calculus of variations.}]
A function $f:D\to(-\infty,\infty]$ with $D\subset\rm$ nonempty convex is said to be quasiconvex if its sublevel set $\{x\in D \mid f(x)\le\alpha\}$ is convex for any $\alpha\in\mathbb{R}$.
\end{dfn}

A further study shows that $\lambda_\mathrm{max}(\cdot,\cdot)$ is actually pseudoconvex on $\sn_{\succ0}\times\sn_{\succ0}$ \cite{nishioka23coap}.
\begin{dfn}[Pseudoconvex function]
A locally Lipschitz continuous function $f:D\to\mathbb{R}$ with $D\subset\rm$ nonempty convex is said to be pseudoconvex if the implication $f(x)>f(y)\ \Rightarrow\ \langle g,y-x \rangle < 0$ holds for any $x,y\in D$ and a Clarke subgradient $g\in\partial f(x)$.
\end{dfn}
Note that the implications ``$\text{convex}\implies\text{pseudoconvex}\implies\text{quasiconvex}$'' hold (see \cite{soleimani07} for details). One advantage of a pseudoconvex function is that it has no (non-global) local optimal solutions, and thus it is possible to compute a global optimal solution in many cases \cite{kiwiel01,nishioka23coap}. We summarize the above results as follows. 
\begin{prp}
The maximum generalized eigenvalue function $\lambda_\mathrm{max}(\cdot,\cdot):\sn_{\succ0}\times\sn_{\succ0}\to\r$ is continuous and pseudoconvex.
\end{prp}

For positive definite matrices, the minimization of the maximum generalized eigenvalue and the maximization of the minimum one with the same two matrices switched is equivalent. Indeed, we can easily show that
\e{
\lambda_\mathrm{min}(X,Y)=\frac{1}{\lambda_\mathrm{max}(Y,X)}
}
holds for any $(X,Y)\in\sn_{\succ0}\times\sn_{\succ0}$. This relationship also holds for the extended definition presented in Section 2.2.2 (see Appendix A).

\subsubsection{Singular positive semidefinite matrices}

The definition \eqref{ge} is not well-defined for singular matrices, since any $\lambda\in\r$ satisfies \eqref{ge} if there exists $v\neq0$ such that $v\in\kernel X\cap\kernel Y$. This is also evident from the fact that the denominator of the Rayleigh quotient \eqref{rayleigh} can be zero if $Y$ is singular. 

Achtziger and Ko{\v{c}}vara \cite{achtziger07siam} extend the definition of the minimum generalized eigenvalue to possibly singular positive semidefinite matrices. They consider two symmetric-matrix-valued functions $K,M:\rm_{\ge0}\to\sn_{\succeq0}$ called the stiffness and mass matrices, respectively, satisfying $\kernel M(x)\subseteq \kernel K(x)$ for any $x\in\rm_{\ge0}\backslash\{0\}$ (see Sections 4 and 5 for details). Then, they define $\lambda_\mathrm{min}(K(\cdot),M(\cdot)):\rm_{\ge0}\backslash\{0\}\to\r$ by
\al{
\lambda_\mathrm{min}(K(x),M(x))
& \coloneq \inf\{\lambda\in\r\mid\exists v\in\rm\backslash\kernel M(x)\text{ s.t. }K(x)v=\lambda M(x)v\}\\
& = \sup\{\lambda\in\r\mid K(x)-\lambda M(x)\succeq0\}\\
& = \underset{v\notin\kernel M(x)}{\inf}\frac{v^\top K(x)v}{v^\top M(x)v}.
\label{ac}
}
We further generalize the definitions of the maximum and minimum generalized eigenvalue with matrix variables in Section 3.

\subsection{Epi-convergence}

Variational analysis provides a convenient tool called the epi-convergence to consider the convergence of optimal solutions of approximated optimization problems to those of the original optimization problems. This tool is especially useful when we consider extended real-valued or discontinuous functions. The following definitions and propositions are all from \cite{rockafellar98}.

The epi-convergence of functions is defined by the set convergence of epigraphs of functions.

\begin{dfn}[Epigraph]
The epigraph of a function $f:\rm\to(-\infty,\infty]$ is defined by
\e{
\mathrm{epi} f\coloneq\{(x,\alpha)\in\rm\times\r\mid f(x)\le\alpha\}.
}
\end{dfn}

Properties of a function and its epigraph are related to each other:
\begin{itemize}
\item A function $f:\rm\to(-\infty,\infty]$ is lower semi-continuous if and only if its epigraph is closed in $\rm\times\r$.
\item A function $f:\rm\to(-\infty,\infty]$ is convex if and only if its epigraph is convex in $\rm\times\r$.
\end{itemize}

\begin{dfn}(Epi-convergence)
A sequence of functions $f^k:\rm\to(-\infty,\infty]$ is said to epi-converge to the epi-limit $f:\rm\to(-\infty,\infty]$ if the sequence of their epigraphs $\mathrm{epi} f^k$ converges to $\mathrm{epi} f$ in the sense of the Painlev\'{e}--Kuratowski set convergence.
\end{dfn}

\begin{dfn}(Painlev\'{e}--Kuratowski set convergence)
Let $(X,d)$ be a metric space. Define $d(x,A)\coloneq\underset{y\in A}{\inf}d(x,y)$ for any $x\in X$ and $A\subseteq X$. When both the outer limit
\e{
\underset{k\to\infty}{\lim\,\sup}A^k\coloneq\{x\in X\mid\underset{k\to\infty}{\lim\,\inf}\ d(x,A^k)\}
}
and the inner limit
\e{
\underset{k\to\infty}{\lim\,\sup}A^k\coloneq\{x\in X\mid\underset{k\to\infty}{\lim\,\sup}\ d(x,A^k)\}
}
are equal to $A^*\subseteq X$, a sequence of sets $A^k$ with $A^k\subseteq X$ is said to converge to $A^*$ in the sense of Painlev\'{e}--Kuratowski.
\end{dfn}

There is also an equivalent definition of epi-convergence without using a notion of set convergence (See \cite[Proposition 7.2]{rockafellar98}). Epi-convergence is often called $\Gamma$-convergence in the theory of calculus of variations and partial differential equations.

Epi-convergence defines the convergence of functions using the convergence of sets (epigraphs). Consequently, epi-convergent functions have nice properties such that certain sets associated with the sequence of functions (optimal solution set, subdifferential set, etc.) also converge to the sets of the epi-limit function in some sense and under certain conditions. The following proposition is the main tool of our analysis, which states that epi-convergence ensures, in some sense, the convergence of the optimal values and solutions.

\begin{prp}[Epi-convergence and convergence in minimization {\cite[Theorem 7.33]{rockafellar98}}]
Suppose that a sequence of functions $f^k:\rm\to(-\infty,\infty]$ has bounded effective domains\footnote{The original condition in \cite[Theorem 7.33]{rockafellar98} is the eventual level boundedness of $\{f^k\}$, which is a weaker assumption than the bounded effective domains. We use the latter assumption for simplicity.} $\domain f^k$, and that $f^k$ epi-converges to $f$ with $f^k$ and $f$ proper and lower semi-continuous. Then
\e{
\underset{k\to\infty}{\lim}\underset{x\in\rm}{\inf} f^k(x) = \underset{x\in\rm}{\inf} f(x)
}
and any accumulation points of $\{x^k\}$ with $x^k\in\underset{x\in\rm}{\mathrm{arg\,min}}f^k(x)$ belong to $\underset{x\in\rm}{\mathrm{arg\,min}}f(x)$. If $\underset{x\in\rm}{\mathrm{arg\,min}}f(x)$ is a singleton, $\{x^k\}$ converges to the minimizer of $f$.
\label{p_conv}
\end{prp}

We use the following proposition to prove the epi-convergence of the proposed approximation.

\begin{prp}[A condition of epi-convergence {\cite[Proposition 7.4(d)]{rockafellar98}}]
If a sequence of functions $f^k:\rm\to(-\infty,\infty]$ is nondecreasing, i.e.~$f^k(x)\le f^{k+1}(x)$ for any $x$, then $f^k$ epi-converges to $\sup_k\,\mathrm{cl} f^k$.
\label{p_cond}
\end{prp}

\begin{dfn}[Closure of function]
The closure of a function $f:\rm\to(-\infty,\infty]$, denoted by $\mathrm{cl}f:\rm\to(-\infty,\infty]$ is defined by
\e{
\mathrm{epi}\,\mathrm{cl}f=\mathrm{cl}\,\mathrm{epi}f,
}
where the second $\mathrm{cl}$ means the closure of the set $\mathrm{epi}f$.
\end{dfn}
The closure of $f$ is also called the lower semi-continuous envelope of $f$.

Note that, in this paper, we consider functions defined only on strict subsets of $\rm$. However, we can apply the definitions and propositions above stated for functions defined on the entire space $\rm$, just by setting the function values $\infty$ outside of the domains.

\section{General results}

We propose an extended definition and an epi-convergent approximation of the maximum generalized eigenvalue of possibly singular semidefinite matrices. All of the results can easily be applied to the minimum generalized eigenvalue; see Appendix \ref{A1}.

\subsection{Extended definition}

We generalize the definitions of the maximum and minimum generalized eigenvalue of possibly singular positive semidefinite matrices. Although our extension is in the same manner as the extended definition introduced in \cite{achtziger07siam}, our definition is more general. We consider a function $\lambda_\mathrm{max}(\cdot,\cdot):\sn_{\succeq0}\times\sn_{\succeq0}\to(-\infty,\infty]$, namely the arguments can be any symmetric positive semidefinite matrices. Then, we consider a composition of $\lambda_\mathrm{max}(\cdot,\cdot)$ with any continuous symmetric-positive-semidefinite-matrix-valued functions $A,B:S\to\sn_{\succeq0}$ defined on $S\subseteq\rm$. 

\begin{dfn}
We define the maximum generalized eigenvalue function $\lambda_\mathrm{max}(\cdot,\cdot):\sn_{\succeq0}\times\sn_{\succeq0}\to(-\infty,\infty]$ by
\e{
\lambda_\mathrm{max}(X,Y)\coloneq
\begin{cases}
\underset{v\notin\kernel Y}{\sup}\frac{v^\top Xv}{v^\top Yv} & (Y\neq0),\\
\infty & (X\neq0,\ Y=0),\\
0 & (X=Y=0).\\
\end{cases}
\label{dfn1}
}
\end{dfn}

\begin{rmk}
The above definition is consistent with Definition \ref{ge} and \eqref{rayleigh} for any $Y\succ0$. The above definition means that we treat the Rayleigh quotient as $\mathrm{const.}/0=\infty$ and $0/0=0$.
\end{rmk}

We have the following equivalent definition, which is useful in our analysis.

\begin{thm}
The maximum generalized eigenvalue function $\lambda_\mathrm{max}(\cdot,\cdot):\sn_{\succeq0}\times\sn_{\succeq0}\to(-\infty,\infty]$ defined by \eqref{dfn1} has the following equivalent definition:
\e{
\lambda_\mathrm{max}(X,Y)=\inf\{\alpha\ge0\mid\alpha Y-X\succeq0\}.
\label{dfn2}
}
Note that if there does not exist $\alpha\ge0$ satisfying $\alpha Y-X\succeq0$, we define $\inf\{\alpha\ge0\mid\alpha Y-X\succeq0\}=\infty$.
Consequently, $\lambda_\mathrm{max}(\cdot,\cdot)$ is nonnegative, lower semi-continuous, and quasiconvex on $\sn_{\succeq0}\times\sn_{\succeq0}$.
\label{t_dfn}
\end{thm}
\begin{proof}
The equivalence is obvious when $X=Y=0$ and $X\neq0,\ Y=0$. We consider the rest cases, i.e.~$Y\neq0$. First, we show
\e{
\inf\{\alpha\ge0\mid\alpha Y-X\succeq0\}\ge\underset{v\notin\kernel Y}{\sup}\frac{v^\top Xv}{v^\top Yv}.
\label{ineq1}
}
Let $\alpha^*=\inf\{\alpha\ge0\mid\alpha Y-X\succeq0\}$. When $\alpha^*=\infty$, \eqref{ineq1} obviously holds. When $\alpha^*<\infty$, we have $\alpha^* Y-X\succeq0$ by the closedness of the positive semidefinite cone and
\al{
\alpha^* Y-X\succeq0
& \iff \alpha^* v^\top Yv \ge v^\top Xv\ \ (\forall v\in\rn)\\
& \implies \alpha^* \ge \frac{v^\top Xv}{v^\top Yv}\ \ (\forall v\notin\kernel Y),
}
which implies \eqref{ineq1}.

Next, we show
\e{
\inf\{\alpha\ge0\mid\alpha Y-X\succeq0\}\le\underset{v\notin\kernel Y}{\sup}\frac{v^\top Xv}{v^\top Yv}.
\label{ineq2}
}
Let 
\e{
\beta^*=\underset{v\notin\kernel Y}{\sup}\frac{v^\top Xv}{v^\top Yv}.
\label{eq1}
}
When $\beta^*=\infty$, \eqref{ineq2} obviously holds. We show that $\kernel Y\subseteq\kernel X$ if $\beta^*<\infty$ by contradiction. Suppose that there exists $v$ such that $v\in\kernel Y$ and $v\notin\kernel X$. By taking a sequence $\{u_k\}$ of vectors such that $u_k\perp\kernel Y$ (note that $Y\neq0$), $u_k\neq0$, and $\|u_k\|\to0$ and setting $v_k=v+u_k\notin\kernel Y$, we obtain
\e{
\frac{v_k^\top Xv_k}{v_k^\top Yv_k}\ge\frac{v^\top Xv}{u_k^\top Yu_k}\to\infty,
}
which contradicts to $\beta^*<\infty$. Now we have $\kernel Y\subseteq\kernel X$, which implies $v^\top(\beta^*Y-X)v=0$ holds for any $v\in\kernel Y$. Also, \eqref{eq1} implies $v^\top(\beta^*Y-X)v\ge0$ holds for any $v\notin\kernel Y$. Thus, $\beta^*Y-X\succeq0$ holds, which implies \eqref{ineq2}.

The above discussions imply that the equality holds for \eqref{ineq1} and \eqref{ineq2}, which shows the equivalence of \eqref{dfn1} and \eqref{dfn2}.

The nonnegativity is obvious from \eqref{dfn2}. The lower semi-continuity and quasiconvexity of $\lambda_\mathrm{max}(\cdot,\cdot):\sn_{\succeq0}\times\sn_{\succeq0}\to[0,\infty]$ immediately follows by the fact that the sublevel set
\e{
\{(X,Y)\in\sn_{\succeq0}\times\sn_{\succeq0}\mid\lambda_\mathrm{max}(X,Y)\le\alpha\}=\{(X,Y)\in\sn_{\succeq0}\times\sn_{\succeq0}\mid\alpha Y-X\succeq0\}
}
is closed and convex for any $\alpha\in\r$ (see e.g.~\cite[Theorem 1.6]{rockafellar98}).
\end{proof}

A simple observation of fixing $Y=0$ and $X\to0$ in \eqref{dfn1} shows that $\lambda_\mathrm{max}(\cdot,\cdot)$ is disconituous.

\subsection{Epi-convergent approximation}

We consider the function
\e{
(X,Y)\mapsto\lambda_\mathrm{max}(X,Y+\epsilon I)
}
defined on $\sn_{\succeq0}\times\sn_{\succeq0}$, which is an epi-convergent approximation of $(X,Y)\mapsto\lambda_\mathrm{max}(X,Y)$ as the following theorem suggests. Namely, it is an approximation that epi-converges to the original function and is real-valued and continuous, hence easier to treat numerically. Note that the technique of adding $\epsilon I$ to possibly singular semidefinite matrices to make it regular is often used in matrix analysis. See \cite{bhatia13} for example.

\begin{thm}
The function $\lambda_\mathrm{max}(\cdot,\cdot+\epsilon I):\sn_{\succeq0}\times\sn_{\succeq0}\to\r$ with $\epsilon>0$ is continuous and epi-converges to $\lambda_\mathrm{max}(\cdot,\cdot):\sn_{\succeq0}\times\sn_{\succeq0}\to(-\infty,\infty]$ defined by \eqref{dfn1}.
\label{t_epi}
\end{thm}
\begin{proof}
The continuity of $\lambda_\mathrm{max}(\cdot,\cdot+\epsilon I)$ follows by the same argument of the continuity of $\lambda_\mathrm{max}(\cdot,\cdot)$ on $\sn_{\succ0}\times\sn_{\succ0}$ discussed in Section 2.2.

For any $X,Y\in\sn_{\succeq0}$ and $0<\epsilon\le\epsilon'$, we have
\e{
\alpha Y-X+\epsilon I\succeq 0 \implies \alpha Y-X+\epsilon' I\succeq 0,
}
hence, by the definition \eqref{dfn2},
\e{
\lambda_\mathrm{max}(X,Y+\epsilon I)\ge\lambda_\mathrm{max}(X,Y+\epsilon' I).
}
Therefore, by Proposition \ref{p_cond}, $(X,Y)\mapsto\lambda_\mathrm{max}(X,Y+\epsilon I)$ epi-converges to $(X,Y)\mapsto\sup_{\epsilon>0}\mathrm{cl}\lambda_\mathrm{max}(X,Y+\epsilon I)$, which is equal to $(X,Y)\mapsto\lim_{\epsilon\to0}\lambda_\mathrm{max}(X,Y+\epsilon I)$ due to the continuity with respect to $(X,Y)$ and the monotonicity with respect to $\epsilon$.

We show that $\lim_{\epsilon\to0}\lambda_\mathrm{max}(X,Y+\epsilon I)=\lambda_\mathrm{max}(X,Y)$ for any $X,Y\in\sn_{\succeq0}$. 

For any $X,Y\in\sn_{\succeq0}$ such that there exists $u$ satisfying $u\notin\kernel X$ and $u\in\kernel Y$ (including the case with $X\neq0$ and $Y=0$), we have $\lambda_\mathrm{max}(X,Y)=\infty$ since $u^\top(\alpha Y-X)u=-u^\top Xu<0$ implies $\alpha Y-X\nsucceq0$ for any $\alpha\ge0$. Moreover, by using such $u$ with the unit norm, which is independent of $\epsilon$, we obtain
\al{
\lambda_\mathrm{max}(X,Y+\epsilon I)
& = \underset{\|v\|=1}{\sup}\frac{v^\top Xv}{v^\top (Y+\epsilon I)v}\\
& \ge \frac{u^\top Xu}{u^\top (Y+\epsilon I)u}\\
& = \frac{u^\top Xu}{\epsilon}\\
& \to \infty\ \ (\epsilon\to0).
}

For any $X,Y\in\sn_{\succeq0}$ such that there does not exist $u$ satisfying $u\notin\kernel X$ and $u\in\kernel Y$ (this is the case when $\kernel Y\subseteq\kernel X$ holds including when $Y$ is regular), we have $\lambda_\mathrm{max}(X,Y)<\infty$ since $u^\top(\alpha Y-X)u=0$ for $u\in\kernel Y$ and $u^\top(\alpha Y-X)u\ge0$ for $u\notin\kernel Y$ and sufficiently large $\alpha\ge0$. We consider the two cases separately. (i) In the special case with $X=Y=0$, we have $\lambda_\mathrm{max}(0,0)=\lambda_\mathrm{max}(0,\epsilon I)=0$ for any $\epsilon>0$. (ii) In the rest case ($\kernel Y\subseteq\kernel X$ and $Y\neq0$) by continuity of $v\mapsto(v^\top Xv)/(v^\top Yv)$ on $\rm\backslash\kernel Y$, for any $\delta>0$, there exists $u\notin\kernel Y$ (near-optimal solution) such that
\e{
\lambda_\mathrm{max}(X,Y)\le\frac{u^\top Xu}{u^\top Yu}+\delta.
}
Thus, we obtain
\al{
0
& \le \lambda_\mathrm{max}(X,Y)-\lambda_\mathrm{max}(X,Y+\epsilon I)\\
& \le \frac{u^\top Xu}{u^\top Yu}+\delta-\frac{u^\top Xu}{u^\top (Y+\epsilon I)u}\\
& \to \delta\ \ (\epsilon\to0).
}
Since $\delta>0$ is arbitarary, we obtain $\lambda_\mathrm{max}(X,Y+\epsilon I)\to\lambda_\mathrm{max}(X,Y)\ (\epsilon\to0)$.

Therefore, for any fixed $X,Y\in\sn_{\succeq0}$, we obtain $\lim_{\epsilon\to0}\lambda_\mathrm{max}(X,Y+\epsilon I)=\lambda_\mathrm{max}(X,Y)$, i.e.~$\lambda_\mathrm{max}(\cdot,\cdot+\epsilon I)$ epi-converges to $\lambda_\mathrm{max}(\cdot,\cdot)$.
\end{proof}

\begin{rmk}
By Theorem \ref{t_epi}, we obtain the third equivalent definition of $(X,Y)\mapsto\lambda_\mathrm{max}(X,Y)$ defined by the epi-limit of $(X,Y)\mapsto\lambda_\mathrm{max}(X,Y+\epsilon I)$ when $\epsilon\to0$.
\end{rmk}

We obtain the following corollary for a composite function of the maximum generalized eigenvalue and symmetric-positive-semidefinite-matrix-valued functions.

\begin{crl}
Let $S\subseteq\rm$ be closed and $A,B:S\to\sn_{\succeq0}$ be continuous. The function $\lambda_\mathrm{max}(A(\cdot),B(\cdot)):S\to(-\infty,\infty]$ is nonnegative and lower semi-continuous on $S$. The approximation $\lambda_\mathrm{max}(A(\cdot),B(\cdot)+\epsilon I):S\to\r$ is continuous and epi-converges to $\lambda_\mathrm{max}(A(\cdot),B(\cdot))$ on $S$. Furthermore, if $A,B:\rm_{\ge0}\to\sn_{\succeq0}$ are affine, $\lambda_\mathrm{max}(A(\cdot),B(\cdot))$ is quasiconvex and $\lambda_\mathrm{max}(A(\cdot),B(\cdot)+\epsilon I)$ is pseudoconvex on $\rm_{\ge0}$.
\label{crl1}
\end{crl}
\begin{proof}
The epi-convergence of $\lambda_\mathrm{max}(A(\cdot),B(\cdot)+\epsilon I)$ can be shown in completely the same way as in Theorem \ref{t_epi}, by replacing $X,Y$ with $A(x),B(x)$. The rest is obvious from the discussion in Section 2.1 and Theorem \ref{t_dfn}.
\end{proof}

\section{Robust compliance optimization}

\subsection{Existing formulation}

We consider the robust (or worst-case) compliance minimization in topology optimization \cite{bental97,kanno15,nishioka23orl}. The compliance is a measure of the deformation of a structure caused by a static applied load. By minimizing it, we can obtain a structure with high stiffness for a given load. The design variable of topology optimization $x\in\rm$ consists of the cross-sectional areas of bars in truss topology optimization and the densities\footnote{$x_j=0$ means the $j$-th element is void and $x_j=1$ means it is filled with the material. We also penalize the intermediate values $(0,1)$ by the SIMP method, etc. See \cite{andreassen11,bendsoe04} for details.} of finite elements in continuum topology optimization, for example.

Assume the design variables are positive, i.e.~$x\in\rm_{>0}$. When we apply a load $p\in\rn$ to a structure $x$, the displacement vector $u\in\rn$ satisfies the equilibrium equation
\e{
K(x)u=p,
\label{equi}
}
where $K(x)\in\sn_{\succeq 0}$ is the global stiffness matrix and $n$ is the number of the degree of freedom of the nodal displacements. When $x\in\rm_{>0}$, $K(x)$ is positive definite and \eqref{equi} is solvable ($u$ uniquely exists). For a truss structure, the global stiffness matrix is a linear function ($K(x)=\sum_{j=1}^m x_j K_j$ with constant symmetric matrices called element stiffness matrices $K_j\in\sn_{\succeq 0}\ (j=1,\ldots,m)$) \cite{kanno15}. For a continuum structure, with the SIMP penalty method, we have $K(x)=\sum_{j=1}^m x_j^q K_j$ with typical penalty parameter $q=3$ \cite{bendsoe04}. We only assume the following for our analysis unless otherwise mentioned.

\begin{asm}
$K:\rm_{\ge0}\to\sn_{\succeq0}$ is continuous and $K(x)\succ0$ for any $x\in\rm_{>0}$.
\label{a_robust}
\end{asm}

The compliance of a structure $x\in\rm_{>0}$ and an applied load $p\in\rn$ is defined by the work done by the load:
\e{
p^\top u = p^\top K(x)^{-1}p.
\label{comp}
}

In robust compliance optimization, we consider the uncertain load $p=Q\hat{p}$ where $\hat{p}\in\rd$ is the uncertain vector satisfying $\|\hat{p}\|=1$, $d$ is the dimension of the uncertain set (generally $n\le d$), and a nonzero constant matrix $Q\in\mathbb{R}^{n\times d}$ is a weight matrix for the uncertainty load. The robust (or worst-case) compliance under the uncertain load \cite{kanno15,takezawa11} is defined for $x\in\rm_{>0}$ by 
\al{
\psi(x)
& \coloneq \underset{\|\hat{p}\|=1}{\max}\hat{p}^\top Q^\top K(x)^{-1}Q\hat{p}\\
& = \lambda_\mathrm{max}(Q^\top K(x)^{-1}Q).
\label{r_eig}
}
It is a nondifferentiable convex function. Note that it includes a standard compliance as a trivial case $d=1$.

The robust compliance can also be formulated as the maximum generalized eigenvalue \cite{cherkaev08,kanno15} for any $x\in\rm_{>0}$:
\e{
\psi(x) = \lambda_\mathrm{max}(QQ^\top,K(x)).
\label{r_gen}
}

The above two formulations are only valid for regular $K(x)$, e.g.~when $x\in\rm_{>0}$. In contrast, semidefinite constraint formulation \cite{bental97}
\e{
\inf\cur{\alpha\ge0\middle|\begin{pmatrix}
K(x) & Q \\
Q^\top & \alpha I \\
\end{pmatrix}
\succeq 0}
\label{r_sdp}
}
is well-defined for any $x\in\rm_{\ge0}$ and equivalent to \eqref{r_eig} and \eqref{r_gen} for any $x\in\rm_{>0}$. Equivalences are consequences of the Schur complement. See \cite{kanno15} for details.

We can also use the extended definition of the maximum generalized eigenvalue \eqref{dfn1} to define the robust compliance for $x\in\rm_{\ge0}$.

\subsection{Extended formulation}

\begin{figure}[t]
  \centering
  \begin{tabular}{cccc}
  \begin{minipage}[t]{0.23\hsize}
    \centering
    \includegraphics[width=1.6cm]{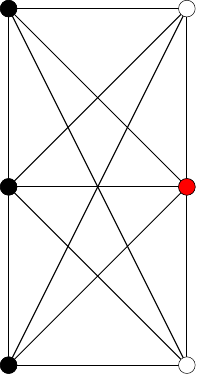}
    \subcaption{$K(x)$ is regular. $x_j>0$ for all $j\in\{1,\ldots,m\}$.}
  \end{minipage} &
  \hspace{-3mm}
  \begin{minipage}[t]{0.23\hsize}
    \centering
    \includegraphics[width=1.6cm]{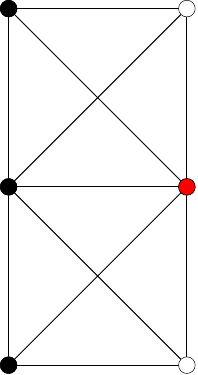}
    \subcaption{$K(x)$ is regular. $x_j=0$ for some $j\in\{1,\ldots,m\}$.}
  \end{minipage} &
  \hspace{-3mm}
  \begin{minipage}[t]{0.23\hsize}
    \centering
    \includegraphics[width=1.6cm]{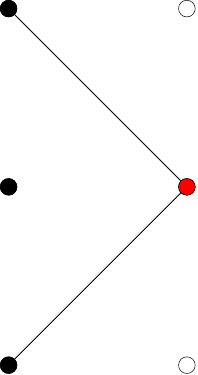}
    \subcaption{$K(x)$ is singular. $\image Q\subseteq\image K(x)$ ($\psi(x)<\infty$).}
  \end{minipage} &
  \hspace{-3mm}
  \begin{minipage}[t]{0.23\hsize}
    \centering
    \includegraphics[width=1.6cm]{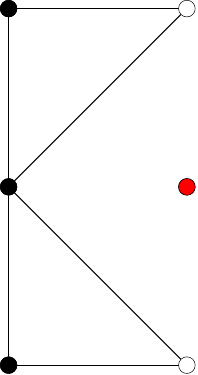}
    \subcaption{$K(x)$ is singular. $\image Q\nsubseteq\image K(x)$ ($\psi(x)=\infty$).}
  \end{minipage}
  \end{tabular}
  \caption{The design of a truss structure and its global stiffness matrix. The black nodes are fixed. An uncertain static load $p$ is applied to the red nodes. The white and red nodes are free to move.}
  \label{f_des}
\end{figure}

Using the extended definition \eqref{dfn1} of the maximum generalized eigenvalue, we redefine the robust compliance $\psi:\rm_{>0}\to\r$ as an extended real-valued function $\psi:\rm_{\ge0}\to(-\infty,\infty]$ by
\e{
\psi(x)\coloneq\lambda_\mathrm{max}(QQ^\top,K(x)),
\label{robust_obj}
}
where $\lambda_\mathrm{max}(\cdot,\cdot)$ is defined by \eqref{dfn1}. The definition \eqref{robust_obj} is consistent with the definition \eqref{r_eig} for any $x\in\rm_{>0}$. As shown in the following proposition, the above definition is equivalent to the definition in \cite{bental97} (i.e.~the second expression of \eqref{r2}).

\begin{prp}
$\psi$ defined by \eqref{robust_obj} has the following equivalent definitions:
\al{
\psi(x)
& = \inf\cur{\alpha\ge0\middle|\begin{pmatrix}
K(x) & Q \\
Q^\top & \alpha I \\
\end{pmatrix}
\succeq 0}\\
& = \underset{\|\hat{p}\|=1}{\sup}\underset{v\in\rn}{\sup}\cur{2\hat{p}^\top Q^\top v-v^\top K(x)v}\\
& =
\begin{cases}
\lambda_\mathrm{max}(Q^\top U)\ \mathrm{with}\ K(x)U=Q & (x\in\mathcal{K}_0),\\
\infty & (\mathrm{otherwise}),
\end{cases}
\label{r2}
}
where 
\al{
\mathcal{K}_0 
& \coloneq \{x\in\rm_{\ge 0}\mid \image Q\subseteq\image K(x)\}\\
& = \{x\in\rm_{\ge 0}\mid \image Q\perp\kernel K(x)\}
\label{robust_set}
}
is the subset of $\rm_{\ge 0}$ such that $K(x)u=Q\hat{p}$ has at least one solution for any $\|\hat{p}\|=1$.
\end{prp}
\begin{proof}
We have
\al{
& \lambda_\mathrm{max}(QQ^\top,K(x))\\
& = \inf\{\alpha\ge0\mid\alpha K(x)-QQ^\top\succeq0\}\\
& = \inf\cur{\alpha\ge0\middle|\begin{pmatrix}
K(x) & Q \\
Q^\top & \alpha I \\
\end{pmatrix}
\succeq 0}\\
& = \inf\cur{\alpha\ge0\middle|(v^\top -\hat{p}^\top) 
\begin{pmatrix}
K(x) & Q \\
Q^\top & \alpha I
\end{pmatrix}
\begin{pmatrix}
v \\
-\hat{p}
\end{pmatrix}
\ge 0\ \ (\forall v\in\rn,\ \hat{p}\in\rd)}\\
& = \inf\cur{\alpha\ge0\mid \alpha\|\hat{p}\|^2 - 2v^\top Q\hat{p} + v^\top K(x)v\ge0\ \ (\forall v\in\rn,\ \|\hat{p}\|=1)}\\
& = \inf\cur{\alpha\ge0\mid \alpha\ge2v^\top Q\hat{p}-v^\top K(x)v\ \ (\forall v\in\rn,\ \|\hat{p}\|=1)}\\
& = \underset{\|\hat{p}\|=1}{\sup}\underset{v\in\rn}{\sup}\cur{2\hat{p}^\top Q^\top v-v^\top K(x)v}
}
where the second equality follows by the Schur complement (see \cite{kanno15}), the fourth equality follows by expanding the quadratic form and ignoring the case $\hat{p}=0$ where the inequality trivially holds, and the fifth equality follows by dividing the inequality by $\|\hat{p}\|^2\neq0$ and replace $\hat{p}/\|\hat{p}\|$ and $v/\|\hat{p}\|$ by $\hat{p}$ and $v$, respectively.

We show the last equality in \eqref{r2}. Suppose $x\in\mathcal{K}_0$. There exists $u\in\rn$ satisfying $K(x)u=Q\hat{p}$ for any $\hat{p}\in\rd$. Since $K(x)u=Q\hat{p}$ is a necessary and sufficient optimality condition for $\sup_{v\in\rd}\cur{2\hat{p}^\top Q^\top v-v^\top K(x)v}$, $u$ attains the supremum. Also, there exists $U\in\r^{n\times d}$ satisfying $K(x)U=Q$ and $u=U\hat{p}+w$ with $w\in\kernel K(x)$. Therefore, we obtain
\al{
\psi(x)
& = \underset{\|\hat{p}\|=1}{\sup}\cur{2\hat{p}^\top Q^\top u-u^\top K(x)u}\\
& = \underset{\|\hat{p}\|=1}{\sup}\cur{2\hat{p}^\top Q^\top (U\hat{p}+w)-(U\hat{p}+w)^\top K(x)(U\hat{p}+w)}\\
& = \underset{\|\hat{p}\|=1}{\sup}\hat{p}^\top Q^\top U\hat{p}\\
& = \lambda_\mathrm{max}(Q^\top U),
}
where the second equality follows by $K(x)u=Q\hat{p}$ and $\image Q\perp\kernel K(x)$. Note that the symmetry of $Q^\top U=U^\top K(x)U$ is used in the last equality.

Suppose $x\notin\mathcal{K}_0$. There exist $p^*$ ($\|p^*\|=1$) and $w\in\kernel K(x)$ such that $p^{*\top} Q^\top w\neq 0$. Since
\al{
\psi(x)
& \ge 2p^{*\top} Q^\top w-w^\top K(x)w\\
& = 2p^{*\top} Q^\top w
}
and the norm of $w$ is arbitrary, we obtain $\psi(x)=\infty$.
\end{proof}

\begin{prp}
Let $K:\rm_{\ge0}\to\sn_{\succeq0}$ satisfy Assumption \ref{a_robust}. The function $\psi:\rm\to(-\infty,\infty]$ defined by \eqref{robust_obj} is lower semicontinuous. Furthermore, if $K$ is affine, $\psi$ is convex.
\end{prp}
\begin{proof}
The lower semicontinuity follows by Theorem \ref{t_dfn}. The convexity follows by the fact that $\psi$ is a supremum of affine functions over a closed convex set $\rm_{\ge0}$ as shown in the second form of \eqref{r2}.
\end{proof}

Figure \ref{f_des} gives simple examples of the relationships between the global stiffness matrices and the corresponding truss designs.

In the semidefinite constraint formulation, for any $x\notin\mathcal{K}_0$, there does not exist $\alpha\in\r$ such that \eqref{r_sdp} is satisfied. This is consistent with $\psi(x)=\infty$ for any $x\notin\mathcal{K}_0$.

We consider a minimization problem of the robust compliance under the volume constraint:
\e{\ald{
& \underset{x\in\rm_{\ge0}}{\mathrm{minimize}} & & \psi(x)\\
& \mathrm{subject\ to} & & l^\top x\le V_0,
\label{p_robust0}
}}
where the constraint $l^\top x\le V_0$ is the volume (or mass) constraint with $V_0>0$ and $l$ being a constant vector with positive components. Note that the volume constraint is essential since $\psi(x)\ge\psi(x')$ for all $x'$ satisfying $x'_j\ge x_j\ (j=1,\ldots,m)$, and thus $\psi(x)$ has no minimizer on $\rm_{\ge0}$. This property is similar to the function $f(x)=1/x$.

The problem \eqref{p_robust0} is equivalent to the following semidefinite programming (SDP) problem:
\e{\ald{
& \underset{x\in\rm_{\ge0}}{\mathrm{minimize}} & & \alpha\\
& \mathrm{subject\ to} & & l^\top x\le V_0,\\
& & &
\begin{pmatrix}
K(x) & Q \\
Q^\top & \alpha I
\end{pmatrix}
\succeq 0.
\label{p_robust_sdp}
}}
Note that when $K$ is affine, \eqref{p_robust_sdp} is a linear SDP, otherwise, it is a nonlinear SDP (e.g.~topology optimization of continua).
It can be solved efficiently by the interior-point method for linear SDP when $K$ is affine and the problem size is not very large. However, when the problem is nonlinear or large-scale, first-order nonsmooth optimization algorithms for the eigenvalue optimization formulation can be more efficient \cite{ding23,nishioka23smao}.

\subsection{Approximations}

The objective function of \eqref{p_robust0} is an extended real-valued function and computationally hard to treat. In practice, we need to avoid the singularity of $K(x)$ on the feasible set to use solvers for linear equations or eigenvalue problems, and thus we consider approximated problems.

\subsubsection{Conventional approximation}

A conventional method to approximate the problem \eqref{p_robust0} is to set an artificial lower limit $\epsilon>0$ on the optimization variables \cite{bendsoe04}:

\e{\ald{
& \underset{x\in[\epsilon,\infty)^m}{\mathrm{minimize}} & & \lambda_\mathrm{max}(Q^\top K(x)^{-1}Q)\\
& \mathrm{subject\ to} & & l^\top x\le V_0.
\label{p_robust1}
}}

The problem \eqref{p_robust1} is a minimization problem of a real-valued function over a closed convex set. It can be easily solved by projected subgradient method \cite{beck17}, smoothing (accelerated) projected gradient method \cite{nishioka23orl}, etc.

When $K$ is affine (e.g.~truss structure), the conventional approximation has the theoretical guarantee as the following theorem implies.

\begin{thm}
Let $K$ satisfy Assumption \ref{a_robust} and be an affine function. Let $\{x^\epsilon\}$ be a sequence of optimal solutions of the approximated problems \eqref{p_robust1} for $\epsilon>0$. Then, $\lim_{\epsilon\to0}\psi(x^\epsilon)=\psi(x^*)$ and any limit point of $\{x^\epsilon\}$ is an optimal solution $x^*$ of the original problem \eqref{p_robust0}.
\label{t_robust1}
\end{thm}
\begin{proof}
Consider an optimal value function $\gamma:\r\to\r$ defined by
\e{
\gamma(\epsilon) = \underset{x\in[\epsilon,\infty)^m}{\min}\{\psi(x)\mid l^\top x\le V_0\}.
}
Since $\psi(x)$ is convex and constraints are linear, by \cite[Lemma 3.58]{beck17}, $\gamma$ is convex. Since $\gamma$ is convex and finite on $\epsilon\in(-\infty,c]$ with some $c>0$ (we set $\gamma(a)=\gamma(0)$ for any $a<0$), it is continuous on $(-\infty,c]$ by \cite[Theorem 2.22]{beck17}. Thus, $\lim_{\epsilon\to0}\gamma(\epsilon)=\gamma(0)$, which means $\lim_{\epsilon\to0}\psi(x^\epsilon)=\psi(x^*)$ and any limit point of $\{x^\epsilon\}$ is an optimal solution of \eqref{p_robust0}.
\end{proof}

\subsubsection{Proposed approximation}

We introduce an approximation of $\psi$ defined on $\rm_{\ge0}$ as follows:
\al{
\psi^\epsilon(x)
& \coloneq \lambda_\mathrm{max}(QQ^\top,K(x)+\epsilon I)\\
& = \lambda_\mathrm{max}(Q^\top(K(x)+\epsilon I)^{-1}Q).
\label{r_approx}
}

\begin{thm}
Let $K$ satisfy Assumption \ref{a_robust}. The function $\psi^\epsilon$ defined by \eqref{r_approx} is continuous on $\rm_{\ge0}$ for any $\epsilon>0$ and epi-converges to $\psi$ defined by \eqref{robust_obj} when $\epsilon\to0$.
\label{t_robust2}
\end{thm}
\begin{proof}
A direct consequence of Corollary \ref{crl1}
\end{proof}

Consider the following approximated problem:
\e{\ald{
& \underset{x\in\rm_{\ge0}}{\mathrm{minimize}} & & \psi^{\epsilon}(x)\\
& \mathrm{subject\ to} & & l^\top x\le V_0.
\label{p_robust2}
}}

\begin{crl}
Let $K$ satisfy Assumption \ref{a_robust}. Let $\{x^\epsilon\}$ be a sequence of optimal solutions of the approximated problems \eqref{p_robust2} for $\epsilon>0$. Then, $\lim_{\epsilon\to0}\psi^\epsilon(x^\epsilon)=\psi(x^*)$ and any limit point of $\{x^\epsilon\}$ is an optimal solution $x^*$ of the original problem \eqref{p_robust0}.
\end{crl}
\begin{proof}
It immediately follows from the compactness of the feasible set of \eqref{p_robust0}, Proposition \ref{p_conv}, and Theorem \ref{t_robust2}.
\end{proof}

One advantage of the proposed approximated problem \eqref{p_robust2} over the conventional approximated problem \eqref{p_robust1} is that a minimizer can have bars with 0 cross-sectional areas in \eqref{p_robust2}, and thus it is actual topology optimization rather than size optimization, which does not admit a topological change of a solution from the initial design. We show that the proposed approximation can produce a solution with many variables equal to zero by numerical experiments. Moreover, the conventional approach fails to be an appropriate approximated problem in eigenfrequency optimization as suggested in the next section.

\section{Eigenfrequency optimization}

In this section, we consider a maximization problem of the minimum eigenfrequency in topology optimization. It is formulated as a minimization problem of the maximum generalized eigenvalue. However, the maximum generalized eigenvalue can be discontinuous unlike the objective function of the robust compliance optimization, and thus the conventional approximation with the lower bound fails to be an accurate approximation. In contrast, the proposed approximation still has a theoretical guarantee of the convergence of solutions.

\subsection{Existing formulation}

For any $x\in\rm_{>0}$, the minimum (fundamental) eigenfrequency of the structure is a positive square root\footnote{We often identify eigenfrequencies and generalized eigenvalues, rather than the positive square root of it, for simplicity.} of the minimum generalized eigenvalue $\lambda_\mathrm{min}(K(x),M(x))$ where $M(x)$ is the so-called mass matrix, which is also symmetric positive definite for any $x\in\rm_{>0}$ in many cases. The stiffness matrix $K(x)$ is the same as the one in the previous section. There is an important relationship $\kernel M(x)\subseteq\kernel K(x)$ in many cases \cite{achtziger07siam}.

When we consider a truss structure or a continuum with the density method \cite{bendsoe04}, the mass matrix is an affine function $M(x)=M_0+\sum_{j=1}^m x_j M_j$ with element mass matrices $M_j\in\sd_{\succeq 0}\ (j=0,\ldots,m)$ ($M_0$ corresponds to a non-structural mass). Note that we denote by $M(x)$ a mass matrix that includes a non-structural mass $M_0$ (possibly $M_0=0$) unlike \cite{achtziger07siam,nishioka23coap}. We only assume the following for our analysis unless otherwise mentioned.

\begin{asm}
$K,M:\rm_{\ge0}\to\sn_{\succeq0}$ are continuous and $K(x),M(x)\succ0$ for any $x\in\rm_{>0}$.
\label{a_freq}
\end{asm}

Since generalized eigenvalues are not well-defined for positive semidefinite (singular) matrices as mentioned in Section 2.2, Achtziger and Ko\v{c}vara \cite{achtziger07siam} introduce the extended formulation defined by \eqref{ac} and consider a maximization problem of the minimum generalized eigenvalue, which can be solved by nonlinear semidefinite programming
\e{
\ald{
& \underset{x\in\rm_{\ge0},\lambda\in\r}{\mathrm{minimize}} & & \lambda\\
& \mathrm{subject\ to} & & \lambda K(x)-M(x)\succeq 0,\\
& & & l^\top x\le V_0
}
}
or the bisection method with semidefinite programming (see \cite{achtziger07siam,nishioka23coap}). In this paper, we investigate solution methods by first-order optimization algorithms to treat possibly large-scale problems.

\subsection{Extended formulation}

Instead of a maximization problem of the minimum generalized eigenvalue, we consider a minimization problem of the maximum one with two matrices switched to make a comparison with robust compliance. The equivalence is proved in Appendix A. We consider the maximum generalized eigenvalue function $\varphi:\rm_{\ge0}\to(-\infty,\infty]$ defined by
\e{
\varphi(x)\coloneq\lambda_\mathrm{max}(M(x),K(x)),
\label{varphi0}
}
where $\lambda_\mathrm{max}(\cdot,\cdot)$ is defined by \eqref{dfn1}. Compared to the domain $\rm_{\ge0}\backslash\{0\}$ of the minimum generalized eigenvalue in \cite{achtziger07siam}, the domain is extended to $\rm_{\ge0}$.

Since $K(0)=0$ and $M(0)=M_0$, the definition \eqref{dfn1} results in
\e{
\varphi(0)=
\begin{cases}
\infty & (M_0\neq0)\\
0 & (M_0=0)
\end{cases}.
}
With a non-structural mass ($M_0\neq0$), it is natural that no structure ($x=0$) leads to a bad objective value, and $x=0$ is never an optimal solution. On the other hand, without a non-structural mass ($M_0=0$), we have $\varphi(0)=0$. Since $\varphi(x)\ge0$ for any $x$, $x=0$ becomes an optimal solution of the minimization of $\varphi(x)$ over $\rm_{\ge0}$. We can interpret this as follows; when there exists neither non-structural mass nor structure itself, vibration never occurs, and thus it can be considered to be optimal. The maximum of the empty set is normally defined as $-\infty$, but we naturally have $\varphi(x)\ge0$, and thus we have the maximum generalized eigenvalue $\varphi(0)=0$.

To exclude the meaningless solution $x=0$, we consider the problem with the equality volume constraint:
\e{\ald{
& \underset{x\in\rm_{\ge0}}{\mathrm{minimize}} & & \varphi(x)\\
& \mathrm{subject\ to} & & l^\top x= V_0.
\label{p_freq0}
}}
This is justified by the following proposition.

\begin{prp}
An optimal solution of the problem \eqref{p_freq0} exists and is also an optimal solution of the following problem: 
\e{\ald{
& \underset{x\in\rm_{\ge0}\backslash\{0\}}{\mathrm{minimize}} & & \varphi(x)\\
& \mathrm{subject\ to} & & l^\top x\le V_0.
}}
\end{prp}
\begin{proof}
The existence of an optimal solution of \eqref{p_freq0} is trivial from the lower semi-continuity of $\varphi$ (proved later in Proposition \ref{lsc}) and the compactness of the feasible set. We have
\al{
\frac{v^\top M(\alpha x)v}{v^\top K(\alpha x)v}
& = \frac{v^\top (M(x)+(1/\alpha-1)M_0)v}{v^\top K(x)v}\\
& \le \frac{v^\top M(x)v}{v^\top K(x)v}
}
for any $\alpha>1$, $x\neq0$, and $v\neq0$. We define $\bar{x}=V_0x/(l^\top x)$ for any $x\in\rm_{\ge0}\backslash\{0\}$ with $l^\top x\le V_0$, and then $\bar{x}$ satisfies $\varphi(x)\ge\varphi(\bar{x})$ and is a feasible point of \eqref{p_freq0}. Therefore, an optimal solution $x^*$ of \eqref{p_freq0} satisfies $\varphi(x)\ge\varphi(x^*)$ for any $x\in\rm_{\ge0}\backslash\{0\}$ with $l^\top x\le V_0$.
\end{proof}

\begin{prp}
Let $K,M$ satisfy Assumption \ref{a_freq}. The function $\varphi$ defined by \eqref{varphi0} is lower semi-continuous on $\rm_{\ge0}$ and continuous on $\rm_{>0}$. Furthermore, when $K,M$ are affine, $\varphi$ is quasiconvex on $\rm_{\ge0}$.
\label{lsc}
\end{prp}
\begin{proof}
A direct consequence of Theorem \ref{t_dfn}.
\end{proof}

\subsection{Approximations}

Since $\varphi$ may be discontinuous on the boundary of $\rm_{\ge0}$, there is no guarantee that a solution of an approximated problem with an artificial lower bound
\e{\ald{
& \underset{x\in[\epsilon,\infty)^m}{\mathrm{minimize}} & & \varphi(x)\\
& \mathrm{subject\ to} & & l^\top x= V_0
\label{p_freq1}
}}
converges to a solution of the original problem \eqref{p_freq0} (see examples in the next subsection). This is a main difference between the robust compliance optimization and the eigenfrequency optimization. In contrast, the proposed approximation still works in the eigenfrequency optimization.

We introduce the following continuous approximation of $\varphi$ defined on $\rm_{\ge0}$:
\e{
\varphi^\epsilon(x)\coloneq\lambda_\mathrm{max}(M(x),K(x)+\epsilon I)
\label{freq_approx}
}

\begin{rmk}
We give an intuitive description of why the approximation \eqref{freq_approx} works well. When $v\in\kernel M(x)$, the Rayleigh quotient $(v^\top M(x)v)/(v^\top K(x)v)$ is undefined ($0/0$) since $\kernel M(x)\subseteq\kernel K(x)$. Therefore, we want to eliminate $v\in\kernel M(x)$ for the candidates of eigenvectors corresponding to the maximum generalized eigenvalue. By adding $\epsilon I$, $(v^\top M(x)v)/(v^\top (K(x)+\epsilon I)v)$ becomes zero for any $v\in\kernel M(x)$, and thus $v\in\kernel M(x)$ is automatically eliminated.
\end{rmk}

\begin{thm}
The function $\varphi^\epsilon$ defined by \eqref{freq_approx} is continuous on $\rm_{\ge0}$ for any $\epsilon>0$ and epi-converges to $\varphi$ defined by \eqref{varphi0} as $\epsilon\to0$.
\label{t_freq}
\end{thm}
\begin{proof}
A direct consequence of Corollary \ref{crl1}.
\end{proof}

We consider the approximated problem of \eqref{p_freq0}:
\e{
\ald{
& \underset{x\in\rm_{\ge0}}{\mathrm{minimize}} & & \varphi^\epsilon(x)\\
& \mathrm{subject\ to} & & l^\top x= V_0,
\label{p_freq2}
}}
with some small $\epsilon>0$.

\begin{crl}
Let $\{x^\epsilon\}$ be a sequence of optimal solutions of the approximated problems \eqref{p_freq2} for $\epsilon>0$. Then, $\lim_{\epsilon\to0}\psi^\epsilon(x^\epsilon)=\psi(x^*)$ and any limit point of $\{x^\epsilon\}$ is an optimal solution $x^*$ of the original problem \eqref{p_freq0}.
\end{crl}
\begin{proof}
It immediately follows from the compactness of the feasible set of \eqref{p_freq0}, Proposition \ref{p_conv}, and Theorem \ref{t_freq}.
\end{proof}

In the case of the truss structure ($K$ and $M$ are affine), the problem \eqref{p_freq2} is a pseudoconvex optimization problem, and thus we can compute a global optimal solution.

\subsection{Simple examples}\label{s_exm}

We consider two simple examples of eigenfrequency optimization to show that although $\varphi$ can be unbounded and discontinuous, $\varphi^\epsilon$ is continuous and suitably approximates $\varphi$.

\subsubsection{Without non-structural mass}
\begin{figure}[ht]
    \centering
    \includegraphics[width=3cm]{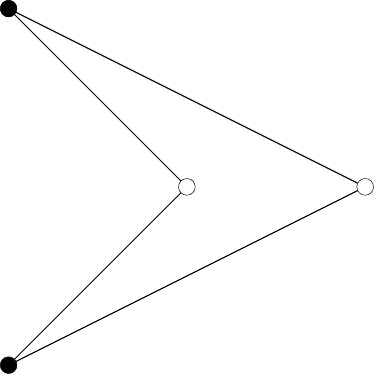}
    \caption{Initial design. The black nodes are fixed. The white nodes are free to move only in the horizontal direction.}
    \label{f_wo_ns}
\end{figure}

\begin{figure}[ht]
  \centering
  \begin{tabular}{cc}
  \begin{minipage}[t]{0.45\hsize}
    \centering
    \includegraphics[width=4.7cm]{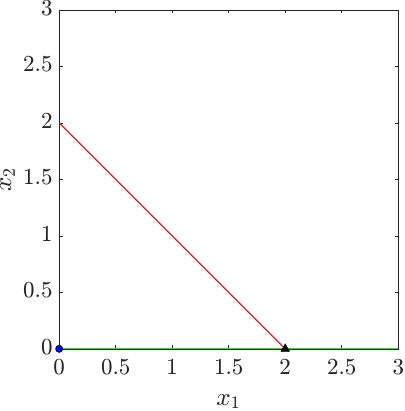}
    \subcaption{The graph of $\varphi$. $\varphi(x)=0$ at the blue dot, $\varphi(x)=1$ on the green line, and $\varphi(x)=2$ elsewhere. The red line is the volume constraint $x_1+x_2=2$. The black triangle is the optimal solution.}
  \end{minipage} &
  \hspace{-2mm}
  \begin{minipage}[t]{0.45\hsize}
    \centering
    \includegraphics[width=5.5cm]{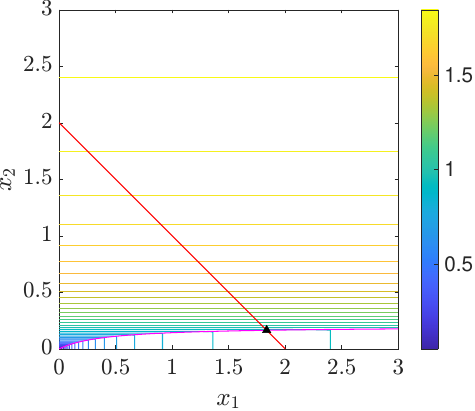}        
    \subcaption{The contour line of $\varphi^\epsilon$ with $\epsilon=0.2$. $\varphi^\epsilon$ is nondifferentiable but continuous on the magenta curve. The black triangle is the optimal solution, which converges to $(x_1,x_2)=(2,0)$ as $\epsilon\to0$.}
  \end{minipage}
  \end{tabular}
  \caption{Graphs of $\varphi(x)$ in \eqref{varphi1} and $\varphi^\epsilon(x)$ in \eqref{varphieps1} (without non-structural mass).}
  \label{f_ex1}
\end{figure}

We consider a simplified version of Example 2.4 in \cite{achtziger07siam}. The initial design is shown in Figure \ref{f_wo_ns}. We can actually construct the following example by scaling lengths, densities, and Young's moduli and restricting the displacements of nodes to one direction in Example 2.4 in \cite{achtziger07siam}. We define
\e{
K(x)=
\begin{pmatrix}
x_1 & 0 \\
0 & x_2 \\
\end{pmatrix},\ \ 
M(x)=
\begin{pmatrix}
x_1 & 0 \\
0 & 2x_2 \\
\end{pmatrix}.
}
Then, we have
\e{
\varphi(x)=
\begin{cases}
2 & (x_1\ge0,\ x_2>0)\\
1 & (x_1>0,\ x_2=0)\\
0 & (x_1=x_2=0)
\end{cases}
\label{varphi1}
}
and
\al{
\varphi^\epsilon(x)
& = \max\cur{\frac{x_1}{x_1+\epsilon},\frac{2x_2}{x_2+\epsilon}}\\
& = 
\begin{cases}
\frac{x_1}{x_1+\epsilon} & \prn{x_2\le\frac{\epsilon x_1}{x_1+2\epsilon}}\\
\frac{2x_2}{x_2+\epsilon} & \prn{x_2>\frac{\epsilon x_1}{x_1+2\epsilon}}
\end{cases}.
\label{varphieps1}
}
As shown in Figure \ref{f_ex1}, $\varphi$ is discontinuous. With the volume constraint $x_1+x_2=2$, the optimal solution of the original problem \eqref{p_freq0} is $(x_1,x_2)=(2,0)$. The optimal solution of the approximated problem \eqref{p_freq1} is $(x_1,x_2)=(1-\epsilon+\sqrt{\epsilon^2-6\epsilon+1},1+\epsilon-\sqrt{\epsilon^2-6\epsilon+1})$, the intersection of $x_1+x_2=2$ and $x_2=\epsilon x_1/(2\epsilon+x_1)$, which converges to $(x_1,x_2)=(2,0)$ when $\epsilon\to0$. In contrast, the approximated problem with the artificial lower bound \eqref{p_freq1} has the optimal solution set $\{(x_1,x_2)\in[\epsilon,\infty)^2\mid x_1+x_2=2\}$, which does not converge to the optimal solution of the original problem $(x_1,x_2)=(2,0)$.

\subsubsection{With non-structural mass}

\begin{figure}[ht]
    \centering
    \includegraphics[width=3cm]{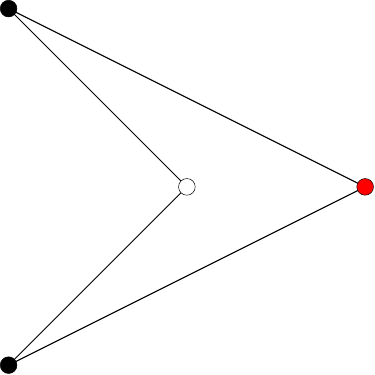}
    \caption{Initial design. The black nodes are fixed. The non-structural mass is applied to the red node. The white and red nodes are free to move only in the horizontal direction.}
    \label{f_w_ns}
\end{figure}

\begin{figure}[ht]
  \centering
  \begin{tabular}{cc}
  \begin{minipage}[t]{0.45\hsize}
    \centering
    \includegraphics[width=5.5cm]{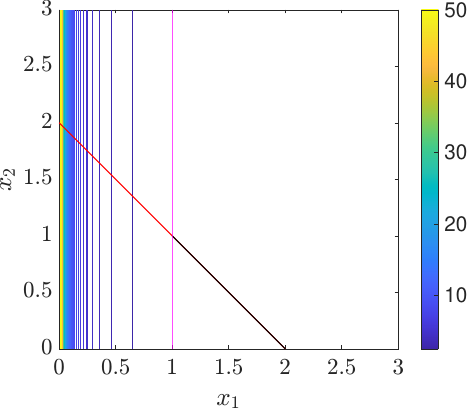}
    \subcaption{The contour line of $\varphi$. $\varphi$ is nondifferentiable but continuous on the magenta line. The red line is the volume constraint $x_1+x_2=2$. The black line is the optimal solution set $\{(x_1,x_2)\in[0,\infty)^2\mid x_1+x_2=2,\ x_1\ge 1\}$.}
  \end{minipage} &
  \hspace{-2mm}
  \begin{minipage}[t]{0.45\hsize}
    \centering
    \includegraphics[width=5.5cm]{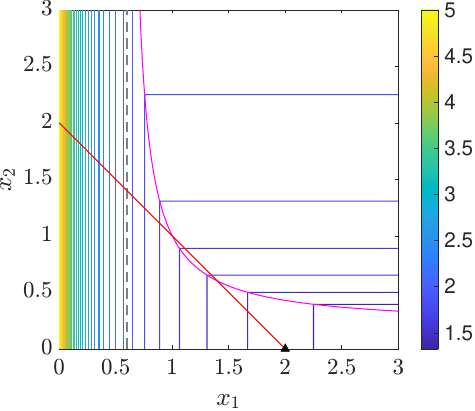}
    \subcaption{The contour line of $\varphi^\epsilon$ with $\epsilon=0.2$. $\varphi^\epsilon$ is nondifferentiable on the magenta curve (the black dashed line is the asymptotic line). The black triangle is the optimal solution $(x_1,x_2)=(2,0)$.}
  \end{minipage}
  \end{tabular}
  \caption{Graphs of $\varphi(x)$ in \eqref{varphi2} and $\varphi^\epsilon(x)$ in \eqref{varphieps2} (with non-structural mass).}
  \label{f_ex2}
\end{figure}

The initial design with a non-structural mass is shown in Figure \ref{f_w_ns}. We have
\e{
K(x)=
\begin{pmatrix}
x_1 & 0 \\
0 & x_2 \\
\end{pmatrix},\ \ 
M(x)=
\begin{pmatrix}
x_1+1 & 0 \\
0 & 2x_2 \\
\end{pmatrix}.
}
Then, we have
\e{
\varphi(x)=
\begin{cases}
1+\frac{1}{x_1} & (0<x_1<1)\\
2 & (x_1\ge1)\\
\infty & (x_1=0)
\end{cases}
\label{varphi2}
}
and
\al{
\varphi^\epsilon(x)
& = \max\cur{\frac{x_1+1}{x_1+\epsilon},\frac{2x_2}{x_2+\epsilon}}\\
& = 
\begin{cases}
\frac{x_1+1}{x_1+\epsilon} & \prn{x_2\le\frac{\epsilon (x_1+1)}{x_1+2\epsilon-1}}\\
\frac{2x_2}{x_2+\epsilon} & \prn{x_2>\frac{\epsilon (x_1+1)}{x_1+2\epsilon-1}}
\end{cases}.
\label{varphieps2}
}
As shown in Figure \ref{f_ex2}, $\varphi$ is unbounded. With the volume constraint $x_1+x_2=2$, the optimal solution set of the original problem \eqref{p_freq0} is $\{(x_1,x_2)\in[0,\infty)^2\mid x_1+x_2=2,\ x_1\ge 1\}$ and the optimal solution of the approximated problem \eqref{p_freq1} is $(x_1,x_2)=(2,0)$, which is in the optimal solution set of the original problem. In this case, the approximated problem with the artificial lower bound \eqref{p_freq1} also works, since the optimal solution set of it is $\{(x_1,x_2)\in[\epsilon,\infty)^2\mid x_1+x_2=2,\ x_1\ge 1\}$, which converges to the optimal solution set of the original problem.

\section{Numerical results}

We compare the proposed approach with several existing approaches on the robust compliance optimization problem and the eigenfrequency optimization problem of truss structures. Note that we do not compare the computational cost since it depends on the problem size and the nonconvexity. The semidefinite programming approach with the interior-point method is usually very efficient in small- to medium-scale convex problems. However, it becomes less efficient for large-scale nonconvex problems. All the experiments in this paper have been conducted on MacBook Pro (2019, 1.4 GHz Quad-Core Intel Core i5, 8 GB memory) and MATLAB R2022b. We used SDPT3 for an SDP solver.

\subsection{Robust compliance opimization}

The initial design (ground structure) is shown in Figure \ref{f_ground_robust}. All nodes are connected by bars and overlapping bars are eliminated. The initial design has uniform cross-sectional areas satisfying the volume constraint.

We set the parameters of the problem as follows; the number of the optimization variables is $m=105$, the size of the matrices is $n=30$, Young's modulus of the material used in the stiffness matrix is $1$ Pa, the distance between the nearest nodes of the truss structure is $1$ m, and the upper limit of the volume is $V_0=0.1$ m$^3$. For details of algorithms, the smoothing accelerated projected gradient method (S-APG) for the robust compliance optimization, see \cite{nishioka23coap,nishioka23orl}.

Figure \ref{f_robust} shows solutions of different formulations of the robust compliance optimization problems \eqref{p_robust_sdp}, \eqref{p_robust1} with $\epsilon=10^{-8}$, and \eqref{p_robust2} with $\epsilon=10^{-4},10^{-9}$. The width of bars in Figure \ref{f_robust} are proportional to the cross-sectional areas of bars. Note that, in practice, we may need to scale the value of $\epsilon$ depending on the value of Young's modulus, etc. As Theorems \ref{t_robust1} and \ref{t_robust2} suggest, solutions of approximated problems \eqref{p_robust1} and \eqref{p_robust2} are close to a solution of the SDP formulation of the original problem \eqref{p_robust_sdp} (A solution of \eqref{p_robust2} becomes closer to a solution of \eqref{p_robust_sdp} with a smaller value of $\epsilon$). One advantage of approximated problems \eqref{p_robust1} and \eqref{p_robust2} is that they can be solved by efficient first-order methods even if they are large-scale, compared to the interior-point method for the original problem \eqref{p_robust_sdp}. Moreover, a projection-based algorithm for the proposed approximated problem \eqref{p_robust2} provides a solution with many bars whose cross-sectional areas are 0, compared to the others. Thick bars are order of $10^{-3}$ and thin bars are order of $10^{-7}$ to $10^{-12}$. Thin bars are possibly due to numerical errors in optimization algorithms. In particular, many non-zero bars of a solution of \eqref{p_robust_sdp} are due to the property of the interior-point method. We may obtain solutions of \eqref{p_robust1} with fewer bars by second-order nonsmooth optimization algorithms. Note that there may be multiple global optimal solutions.

\begin{rmk}
If we set $\epsilon$ too small, numerical instability in the computation of $\psi^\epsilon$ occurs because the matrices are too close to singular. If we set Young's modulus very large, we need to change $\epsilon$ proportional to Young's modulus to avoid the numerical instability.
\end{rmk}

\begin{figure}[ht]
    \centering
    \includegraphics[width=4cm]{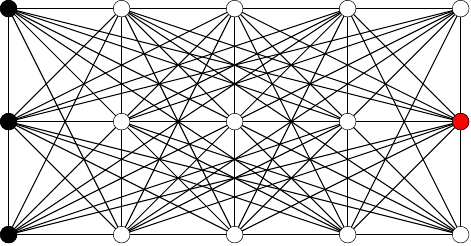}
    \caption{Initial design (ground structure). The black nodes are fixed. The white and red nodes are free to move. The uncertain load is applied on the red node.}
    \label{f_ground_robust}
\end{figure}

\begin{figure}[ht]
  \centering
  \begin{tabular}{cc}
  \begin{minipage}[t]{0.45\hsize}
    \centering
    \includegraphics[width=4cm]{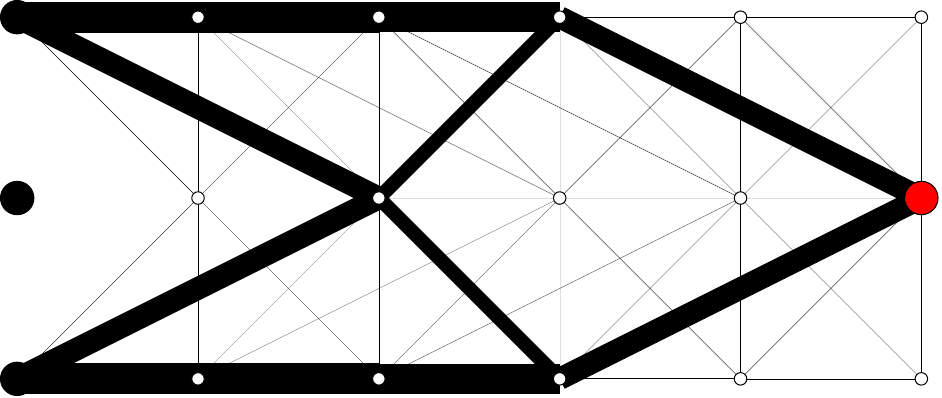}
    \subcaption{A solution of \eqref{p_robust1} with $\epsilon=10^{-9}$ by S-APG ($\psi^\epsilon(x)=220.5246$).}
  \end{minipage} &
  \hspace{-2mm}
  \begin{minipage}[t]{0.45\hsize}
    \centering
    \includegraphics[width=4cm]{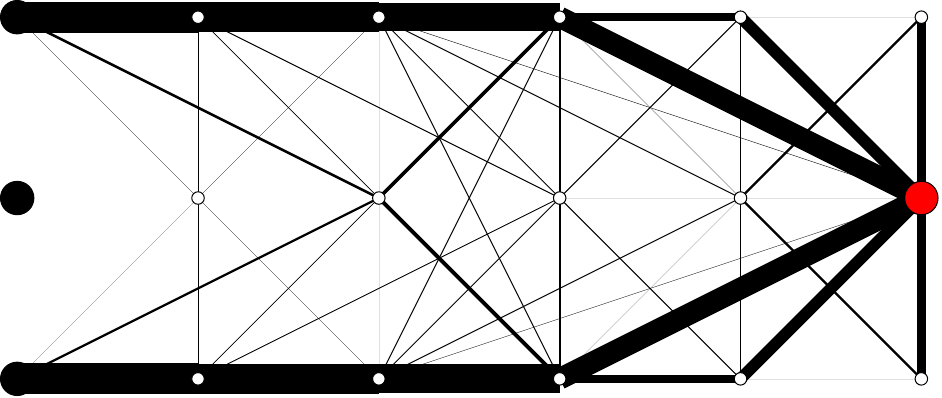}
    \subcaption{A solution of \eqref{p_robust1} with $\epsilon=10^{-4}$ by S-APG ($\psi^\epsilon(x)=65.9391$).}
  \end{minipage}\\
  \begin{minipage}[t]{0.45\hsize}
    \centering
    \includegraphics[width=4cm]{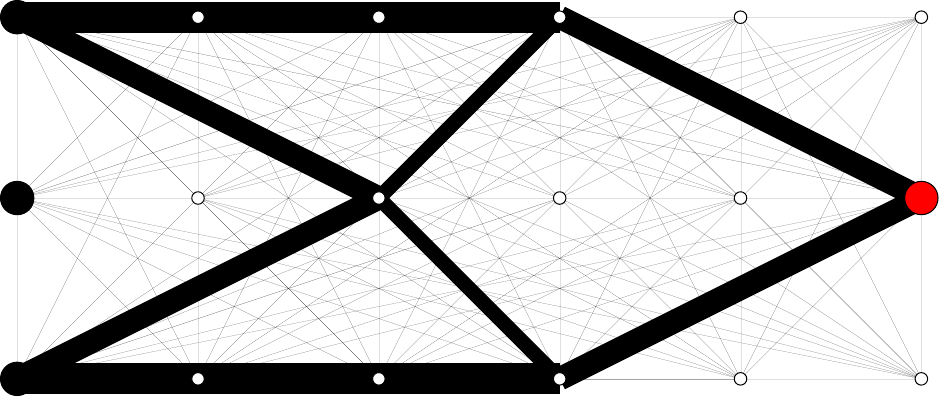}
    \subcaption{A solution of \eqref{p_robust2} by S-APG ($\psi(x)=220.5278$).}
  \end{minipage} &
  \hspace{-2mm}
  \begin{minipage}[t]{0.45\hsize}
    \centering
    \includegraphics[width=4cm]{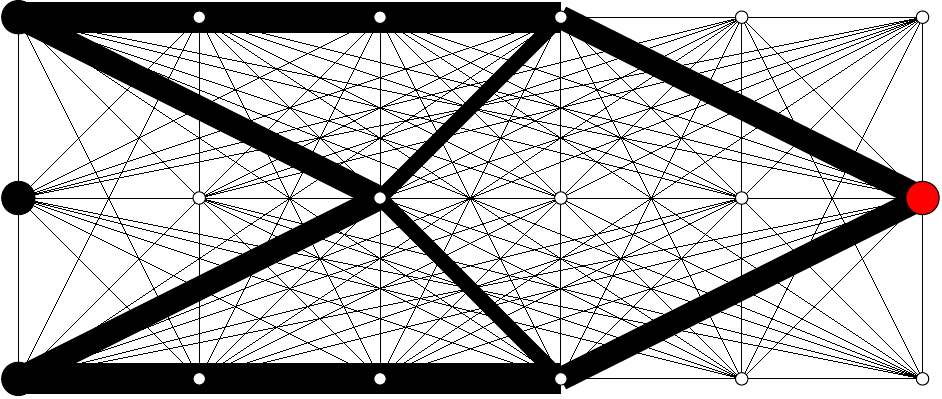}
    \subcaption{A solution of \eqref{p_robust_sdp} by SDPT3 ($\psi(x)=220.5000$).}
  \end{minipage}
  \end{tabular}
  \caption{Solutions of problems \eqref{p_robust1}, \eqref{p_robust2}, and \eqref{p_robust_sdp}. Bars with cross-sectional areas exactly equal to $0$ are not displayed. The black nodes are fixed and the circular uncertain load is applied to the red node.}
  \label{f_robust}
\end{figure}

\subsection{Eigenfrequency optimization}

The initial design (ground structure) is shown in Figure \ref{f_ground_freq}.

We set the parameters of the problem as follows; the number of the optimization variables is $m=200$, the size of the matrices is $n=46$, Young's modulus of the material used in the stiffness matrix is $1$ Pa, the density of the material used in the mass matrix is $7.86\times10^{-8}$ kg/m$^3$, the mass of the non-structural mass is $5\times10^{-5}$ kg, the distance between the nearest nodes is $1$ m, and the upper limit of the volume is $V_0=0.1$ m$^3$. For details of algorithms, the smoothing accelerated projected gradient method (S-APG) and the bisection method, see \cite{nishioka23coap}. Although S-APG does not have the convergence guarantee in the eigenfrequency optimization, it converged the same solution as the subgradient method \cite{kiwiel01}, which has the convergence guarantee to a global optimal solution.

Figure \ref{f_freq} shows solutions of different formulations of the eigenfrequency optimization problems \eqref{p_freq0}, \eqref{p_freq1} with $\epsilon=10^{-8}$, and \eqref{p_freq2} with $\epsilon=10^{-4},10^{-8}$. As Theorems \ref{t_freq} suggests, a solution of the approximated problem \eqref{p_freq2} is close to a solution of the original problem \eqref{p_freq0} (solutions are less sensitive to the value of $\epsilon$ compared to the robust compliance optimization). A solution of \eqref{p_freq1} is also close to a solution of the original problem \eqref{p_freq0}, even though there are no theoretical guarantees.

\begin{figure}[ht]
    \centering
    \includegraphics[width=3cm]{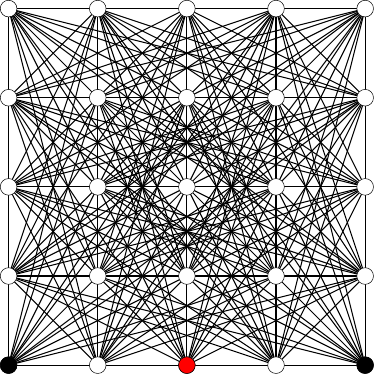}
    \caption{Initial design (ground structure). The black nodes are fixed. The white and red nodes are free to move. The non-structural mass is applied to the red node.}
    \label{f_ground_freq}
\end{figure}

\begin{figure}[ht]
  \centering
  \begin{tabular}{cc}
  \begin{minipage}[t]{0.45\hsize}
    \centering
    \includegraphics[width=3cm]{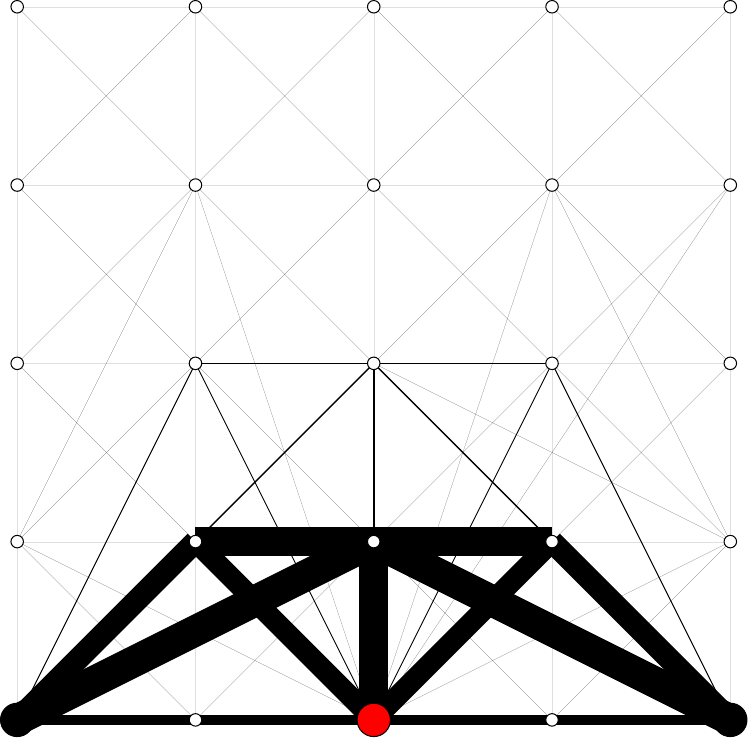}
    \subcaption{A solution of \eqref{p_freq2} with $\epsilon=10^{-8}$ by S-APG ($\psi^\epsilon(x)=0.019455$).}
  \end{minipage} &
  \hspace{-2mm}
  \begin{minipage}[t]{0.45\hsize}
    \centering
    \includegraphics[width=3cm]{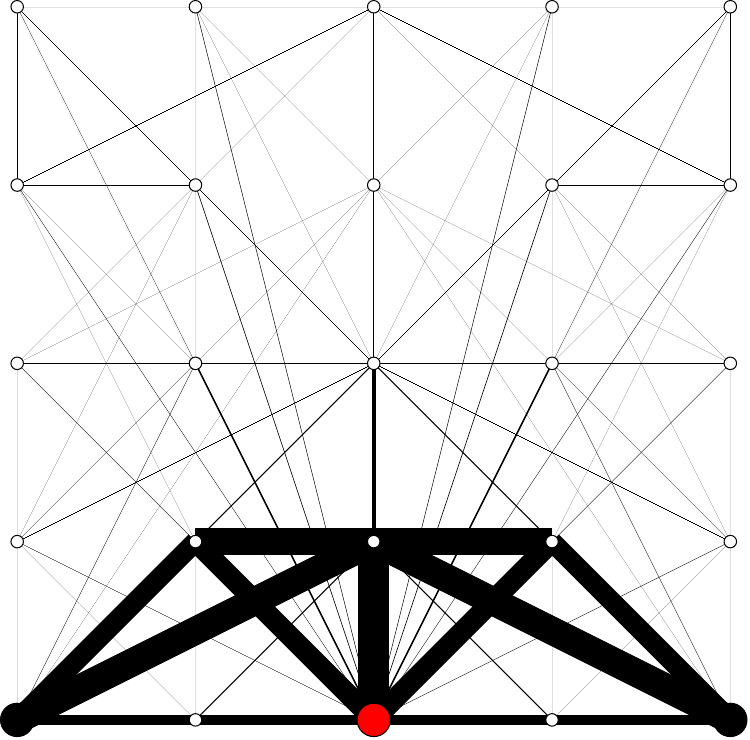}
    \subcaption{A solution of \eqref{p_freq2} with $\epsilon=10^{-4}$ by S-APG ($\psi^\epsilon(x)=0.01681$).}
  \end{minipage}\\
  \begin{minipage}[t]{0.45\hsize}
    \centering
    \includegraphics[width=3cm]{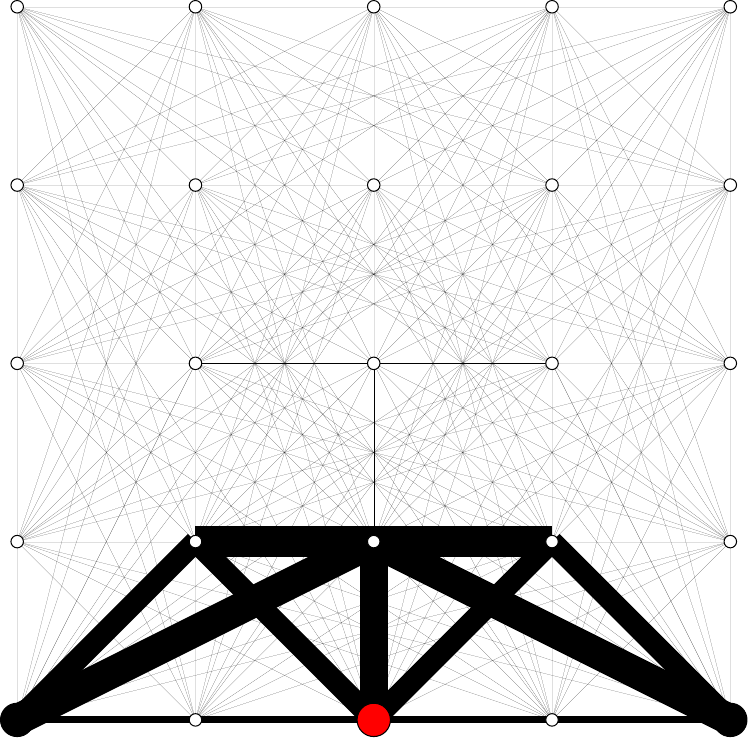}
    \subcaption{A solution of \eqref{p_freq1} by S-APG  ($\psi(x)=0.019455$).}
  \end{minipage} &
  \hspace{-2mm}
  \begin{minipage}[t]{0.45\hsize}
    \centering
    \includegraphics[width=3cm]{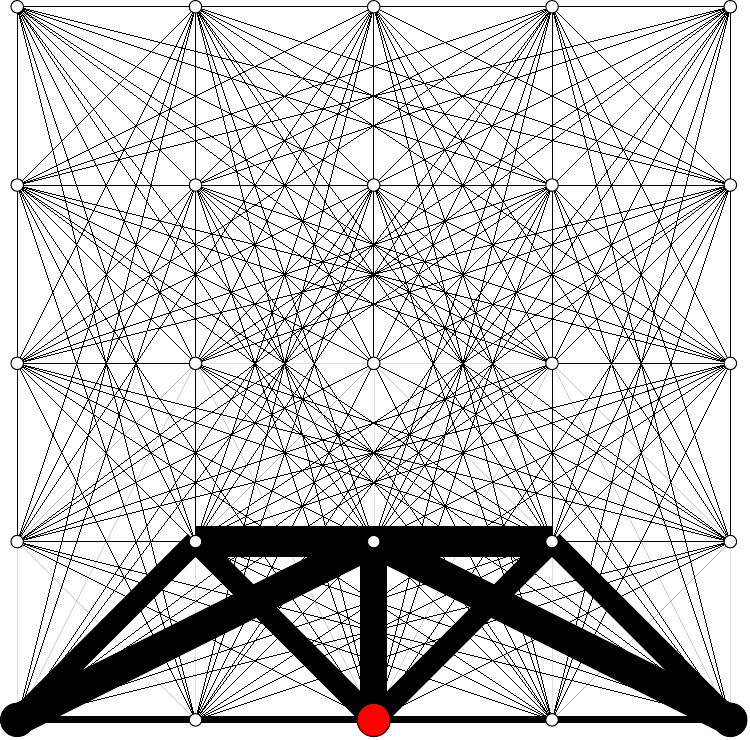}
    \subcaption{A solution of \eqref{p_freq0} by the bisection method with SDPT-3 ($\psi(x)=0.019453$).}
  \end{minipage}
  \end{tabular}
  \caption{Solutions of problems \eqref{p_freq2}, \eqref{p_freq0}, and \eqref{p_freq1}. Bars with cross-sectional areas exactly equal to $0$ are not displayed. The black nodes are fixed and the non-structural mass is applied to the red node.}
  \label{f_freq}
\end{figure}

\section{Conclusion}

We proposed an extended formulation and a new approximation of the maximum (and minimum) generalized eigenvalue function and showed the epi-convergence of the approximation to the proposed extended formulation. We consider two specific examples of generalized eigenvalue optimization problems in topology optimization: the robust compliance optimization and the eigenfrequency optimization. Theoretical and numerical results show that the proposed approximation is a good alternative to the traditional approximation with the artificial lower bound of the design variable.

Although we only conduct numerical experiments on truss topology optimization, theoretical results still hold for continuum topology optimization. Note that continuum topology optimization problems are generally nonconvex and may have multiple local minima and our results can only be applied to the global minimum.

The application of variational analysis to topology optimization is not very common. This paper suggests that tools in variational analysis, especially epi-convergence, can be useful for constructing effective approximations for many complex problems with nonsmoothness, unboundedness, and discontinuity in topology optimization.


\backmatter

\bmhead{Acknowledgments}

This research is part of the results of Value Exchange Engineering, a joint research project between R4D, Mercari, Inc. and RIISE. The work of the first author is partially supported by JSPS KAKENHI JP23KJ0383. The work of the second author is partially supported by JSPS KAKENHI JP21K04351 and JP24K07747.

\bmhead{Data availability statement}

The MATLAB codes used in the numerical experiments are available from the corresponding author upon request.

\bmhead{Conflict of interest}

The authors declare that they have no conflict of interest.

\begin{appendices}

\section{Minimum generalized eigenvalue}\label{A1}

We extend the results to the minimum generalized eigenvalue, which has a simple relationship with the maximum generalized eigenvalue.

\begin{dfn}
We define the minimum generalized eigenvalue function $\lambda_\mathrm{min}(\cdot,\cdot):\sn_{\succeq0}\times\sn_{\succeq0}\to(-\infty,\infty]$ by
\e{
\lambda_\mathrm{min}(X,Y)\coloneq
\begin{cases}
\underset{v\notin\kernel Y}{\inf}\frac{v^\top Xv}{v^\top Yv} & (Y\neq0),\\
\infty & (Y=0).
\end{cases}
\label{dfn1a}
}
\end{dfn}

\begin{thm}
The minimum generalized eigenvalue function $\lambda_\mathrm{min}(\cdot,\cdot):\sn_{\succeq0}\times\sn_{\succeq0}\to(-\infty,\infty]$ defined by \eqref{dfn1a} has the following equivalent definition:
\e{
\lambda_\mathrm{min}(X,Y)=\sup\{\alpha\ge0\mid X-\alpha Y\succeq0\}.
\label{dfn2a}
}
Note that if there does not exist $\alpha\ge0$ satisfying $X-\alpha Y\succeq0$, we define $\sup\{\alpha\ge0\mid X-\alpha Y\succeq0\}=0$.
Consequently, $\lambda_\mathrm{min}(\cdot,\cdot)$ is nonnegative, upper semi-continuous, and quasiconcave on $\sn_{\succeq0}\times\sn_{\succeq0}$.
\label{t_dfna}
\end{thm}
\begin{proof}
It is almost identical to the proof of Theorem \ref{t_dfn}. 

The equivalence is obvious when $Y=0$. We consider the rest cases. First, we show
\e{
\sup\{\alpha\ge0\mid X-\alpha Y\succeq0\}\le\underset{v\notin\kernel Y}{\inf}\frac{v^\top Xv}{v^\top Yv}.
\label{ineq1a}
}
Let $\alpha^*=\sup\{\alpha\ge0\mid X-\alpha Y\succeq0\}$. Note that when $Y\neq0$, $\alpha^*<\infty$. We have $X-\alpha^* Y\succeq0$ by the closedness of the positive semidefinite cone and
\al{
X-\alpha^* Y\succeq0
& \iff v^\top Xv\ge\alpha^* v^\top Yv\ \ (\forall v\in\rn)\\
& \implies \alpha^* \le \frac{v^\top Xv}{v^\top Yv}\ \ (\forall v\notin\kernel Y),
}
which implies \eqref{ineq1a}.

Next, we show
\e{
\sup\{\alpha\ge0\mid X-\alpha Y\succeq0\}\ge\underset{v\notin\kernel Y}{\inf}\frac{v^\top Xv}{v^\top Yv}.
\label{ineq2a}
}
Let 
\e{
\beta^*=\underset{v\notin\kernel Y}{\inf}\frac{v^\top Xv}{v^\top Yv}.
\label{eq1a}
}
Note that when $Y\neq0$, $\beta^*<\infty$. Then, \eqref{eq1a} implies $v^\top(X-\beta^*Y)v\ge0$ holds for any $v\notin\kernel Y$. Also, for any $v\in\kernel Y$, $v^\top(X-\beta^*Y)v=v^\top Xv\ge0$ holds. Thus, $X-\beta^*Y\succeq0$ holds, which implies \eqref{ineq2a}. The above discussions imply that the equality holds for \eqref{ineq1a}, which shows the equivalence of \eqref{dfn1a} and \eqref{dfn2a}.

The nonnegativity, the upper semi-continuity, and the quasiconcavity are obvious by the closedness and convexity of the superlevel set.
\end{proof}

The above results are consistent with the definition \eqref{ac} introduced by \cite{achtziger07siam} and extended since we consider the matrix variables and the case with $Y=0$, which results in an extended real-valued function. Note that when $Y\neq0$, $\lambda_\mathrm{min}(X,Y)$ is finite.

\begin{prp}
For any $(X,Y)\in\sn_{\succeq0}\times\sn_{\succeq0}$, we have
\e{
\lambda_\mathrm{max}(X,Y)=\frac{1}{\lambda_\mathrm{min}(Y,X)},
}
with the conventions $1/0=\infty$ and $1/\infty=0$.
\end{prp}
\begin{proof}
We show the claim by showing the sublevel sets of the functions on both sides of the equality coincide for any $\alpha\in\r$. We write $\{(X,Y)\mid\cdot\}$ to denote $\{(X,Y)\in\sn_{\ge0}\times\sn_{\ge0}\mid\cdot\}$ for simplicity. For $\alpha<0$, the sublevel sets are empty. For $\alpha=0$, we have
\e{
\{(X,Y)\mid\lambda_\mathrm{max}(X,Y)\le0\}=\{(X,Y)\mid-X\succeq0\}=\{(X,Y)\mid X=0\}
}
and
\e{
\cur{(X,Y)\middle|\frac{1}{\lambda_\mathrm{min}(Y,X)}\le0}=\{(X,Y)\mid\lambda_\mathrm{min}(Y,X)=\infty\}=\{(X,Y)\mid X=0\},
}
and thus they coincide. For $\alpha>0$, we have
\e{
\{(X,Y)\mid\lambda_\mathrm{max}(X,Y)\le\alpha\}=\{(X,Y)\mid\alpha Y-X\succeq0\}
}
and
\al{
\cur{(X,Y)\middle|\frac{1}{\lambda_\mathrm{min}(Y,X)}\le\alpha}
& = \cur{(X,Y)\middle|\lambda_\mathrm{min}(Y,X)\ge\frac{1}{\alpha}}\\
& = \cur{(X,Y)\middle| Y-\frac{1}{\alpha}X\succeq0}\\
& = \{(X,Y)\mid\alpha Y-X\succeq0\},
}
which completes the proof.
\end{proof}

We can also construct a hypo-convergent\footnote{A sequence of functions $f^k$ hypo-converges to $f$ iff $-f^k$ epi-converges to $-f$.} approximation of the minimum generalized eigenvalue by almost the same argument in Section 3.

\end{appendices}

\bibliography{ref}


\begin{thebibliography}{25}
\ifx \bisbn   \undefined \def \bisbn  #1{ISBN #1}\fi
\ifx \binits  \undefined \def \binits#1{#1}\fi
\ifx \bauthor  \undefined \def \bauthor#1{#1}\fi
\ifx \batitle  \undefined \def \batitle#1{#1}\fi
\ifx \bjtitle  \undefined \def \bjtitle#1{#1}\fi
\ifx \bvolume  \undefined \def \bvolume#1{\textbf{#1}}\fi
\ifx \byear  \undefined \def \byear#1{#1}\fi
\ifx \bissue  \undefined \def \bissue#1{#1}\fi
\ifx \bfpage  \undefined \def \bfpage#1{#1}\fi
\ifx \blpage  \undefined \def \blpage #1{#1}\fi
\ifx \burl  \undefined \def \burl#1{\textsf{#1}}\fi
\ifx \doiurl  \undefined \def \doiurl#1{\url{https://doi.org/#1}}\fi
\ifx \betal  \undefined \def \betal{\textit{et al.}}\fi
\ifx \binstitute  \undefined \def \binstitute#1{#1}\fi
\ifx \binstitutionaled  \undefined \def \binstitutionaled#1{#1}\fi
\ifx \bctitle  \undefined \def \bctitle#1{#1}\fi
\ifx \beditor  \undefined \def \beditor#1{#1}\fi
\ifx \bpublisher  \undefined \def \bpublisher#1{#1}\fi
\ifx \bbtitle  \undefined \def \bbtitle#1{#1}\fi
\ifx \bedition  \undefined \def \bedition#1{#1}\fi
\ifx \bseriesno  \undefined \def \bseriesno#1{#1}\fi
\ifx \blocation  \undefined \def \blocation#1{#1}\fi
\ifx \bsertitle  \undefined \def \bsertitle#1{#1}\fi
\ifx \bsnm \undefined \def \bsnm#1{#1}\fi
\ifx \bsuffix \undefined \def \bsuffix#1{#1}\fi
\ifx \bparticle \undefined \def \bparticle#1{#1}\fi
\ifx \barticle \undefined \def \barticle#1{#1}\fi
\bibcommenthead
\ifx \bconfdate \undefined \def \bconfdate #1{#1}\fi
\ifx \botherref \undefined \def \botherref #1{#1}\fi
\ifx \url \undefined \def \url#1{\textsf{#1}}\fi
\ifx \bchapter \undefined \def \bchapter#1{#1}\fi
\ifx \bbook \undefined \def \bbook#1{#1}\fi
\ifx \bcomment \undefined \def \bcomment#1{#1}\fi
\ifx \oauthor \undefined \def \oauthor#1{#1}\fi
\ifx \citeauthoryear \undefined \def \citeauthoryear#1{#1}\fi
\ifx \endbibitem  \undefined \def \endbibitem {}\fi
\ifx \bconflocation  \undefined \def \bconflocation#1{#1}\fi
\ifx \arxivurl  \undefined \def \arxivurl#1{\textsf{#1}}\fi
\csname PreBibitemsHook\endcsname

\bibitem[\protect\citeauthoryear{Allaire}{2002}]{allaire02}
\begin{bbook}
\bauthor{\bsnm{Allaire}, \binits{G.}}:
\bbtitle{Shape Optimization by the Homogenization Method}.
\bpublisher{Springer},
\blocation{New York}
(\byear{2002})
\end{bbook}
\endbibitem

\bibitem[\protect\citeauthoryear{Bends{\o}e and Sigmund}{2004}]{bendsoe04}
\begin{bbook}
\bauthor{\bsnm{Bends{\o}e}, \binits{M.P.}},
\bauthor{\bsnm{Sigmund}, \binits{O.}}:
\bbtitle{Topology Optimization: Theory, Methods, and Applications (2nd Ed.)}.
\bpublisher{Springer},
\blocation{Berlin}
(\byear{2004})
\end{bbook}
\endbibitem

\bibitem[\protect\citeauthoryear{Achtziger and Ko{\v{c}}vara}{2007}]{achtziger07siam}
\begin{barticle}
\bauthor{\bsnm{Achtziger}, \binits{W.}},
\bauthor{\bsnm{Ko{\v{c}}vara}, \binits{M.}}:
\batitle{Structural topology optimization with eigenvalues}.
\bjtitle{SIAM Journal on Optimization}
\bvolume{18}(\bissue{4}),
\bfpage{1129}--\blpage{1164}
(\byear{2007})
\end{barticle}
\endbibitem

\bibitem[\protect\citeauthoryear{Ferrari and Sigmund}{2019}]{ferrari19}
\begin{barticle}
\bauthor{\bsnm{Ferrari}, \binits{F.}},
\bauthor{\bsnm{Sigmund}, \binits{O.}}:
\batitle{Revisiting topology optimization with buckling constraints}.
\bjtitle{Structural and Multidisciplinary Optimization}
\bvolume{59}(\bissue{5}),
\bfpage{1401}--\blpage{1415}
(\byear{2019})
\end{barticle}
\endbibitem

\bibitem[\protect\citeauthoryear{Ko\v{c}vara}{2002}]{kocvara02}
\begin{barticle}
\bauthor{\bsnm{Ko\v{c}vara}, \binits{M.}}:
\batitle{On the modelling and solving of the truss design problem with global stability constraints}.
\bjtitle{Structural and Multidisciplinary Optimization}
\bvolume{23}(\bissue{3}),
\bfpage{189}--\blpage{203}
(\byear{2002})
\end{barticle}
\endbibitem

\bibitem[\protect\citeauthoryear{Ohsaki et~al.}{1999}]{ohsaki99}
\begin{barticle}
\bauthor{\bsnm{Ohsaki}, \binits{M.}},
\bauthor{\bsnm{Fujisawa}, \binits{K.}},
\bauthor{\bsnm{Katoh}, \binits{N.}},
\bauthor{\bsnm{Kanno}, \binits{Y.}}:
\batitle{Semi-definite programming for topology optimization of trusses under multiple eigenvalue constraints}.
\bjtitle{Computer Methods in Applied Mechanics and Engineering}
\bvolume{180}(\bissue{1-2}),
\bfpage{203}--\blpage{217}
(\byear{1999})
\end{barticle}
\endbibitem

\bibitem[\protect\citeauthoryear{Torii and de~Faria}{2017}]{torii17}
\begin{barticle}
\bauthor{\bsnm{Torii}, \binits{A.J.}},
\bauthor{\bsnm{Faria}, \binits{J.R.}}:
\batitle{Structural optimization considering smallest magnitude eigenvalues: a smooth approximation}.
\bjtitle{Journal of the Brazilian Society of Mechanical Sciences and Engineering}
\bvolume{39},
\bfpage{1745}--\blpage{1754}
(\byear{2017})
\end{barticle}
\endbibitem

\bibitem[\protect\citeauthoryear{Nishioka et~al.}{2023}]{nishioka23coap}
\begin{botherref}
\oauthor{\bsnm{Nishioka}, \binits{A.}},
\oauthor{\bsnm{Toyoda}, \binits{M.}},
\oauthor{\bsnm{Tanaka}, \binits{M.}},
\oauthor{\bsnm{Kanno}, \binits{Y.}}:
On a minimization problem of the maximum generalized eigenvalue: properties and algorithms.
arXiv preprint arXiv:2312.01603
(2023)
\end{botherref}
\endbibitem

\bibitem[\protect\citeauthoryear{Rockafellar and Wets}{1998}]{rockafellar98}
\begin{bbook}
\bauthor{\bsnm{Rockafellar}, \binits{R.T.}},
\bauthor{\bsnm{Wets}, \binits{R.J.-B.}}:
\bbtitle{Variational Analysis}.
\bpublisher{Springer},
\blocation{Heidelberg}
(\byear{1998})
\end{bbook}
\endbibitem

\bibitem[\protect\citeauthoryear{Boyd and El~Ghaoui}{1993}]{boyd93}
\begin{barticle}
\bauthor{\bsnm{Boyd}, \binits{S.}},
\bauthor{\bsnm{El~Ghaoui}, \binits{L.}}:
\batitle{Method of centers for minimizing generalized eigenvalues}.
\bjtitle{Linear Algebra and Its Applications}
\bvolume{188},
\bfpage{63}--\blpage{111}
(\byear{1993})
\end{barticle}
\endbibitem

\bibitem[\protect\citeauthoryear{Nesterov and Nemirovskii}{1995}]{nesterov95}
\begin{barticle}
\bauthor{\bsnm{Nesterov}, \binits{Y.E.}},
\bauthor{\bsnm{Nemirovskii}, \binits{A.S.}}:
\batitle{An interior-point method for generalized linear-fractional programming}.
\bjtitle{Mathematical Programming}
\bvolume{69},
\bfpage{177}--\blpage{204}
(\byear{1995})
\end{barticle}
\endbibitem

\bibitem[\protect\citeauthoryear{Harville}{1997}]{harville97}
\begin{bbook}
\bauthor{\bsnm{Harville}, \binits{D.A.}}:
\bbtitle{Matrix Algebra From a Statistician's Perspective}.
\bpublisher{Springer},
\blocation{New York}
(\byear{1997})
\end{bbook}
\endbibitem

\bibitem[\protect\citeauthoryear{Bhatia}{2013}]{bhatia13}
\begin{bbook}
\bauthor{\bsnm{Bhatia}, \binits{R.}}:
\bbtitle{Matrix Analysis}.
\bpublisher{Springer},
\blocation{New York}
(\byear{2013})
\end{bbook}
\endbibitem

\bibitem[\protect\citeauthoryear{Morrey}{1952}]{morrey52}
\begin{barticle}
\bauthor{\bsnm{Morrey}, \binits{C.B.}}:
\batitle{Quasi-convexity and the lower semicontinuity of multiple integrals}.
\bjtitle{Pacific Journal of Mathematics}
\bvolume{2},
\bfpage{25}--\blpage{53}
(\byear{1952})
\end{barticle}
\endbibitem

\bibitem[\protect\citeauthoryear{Soleimani-Damaneh}{2007}]{soleimani07}
\begin{barticle}
\bauthor{\bsnm{Soleimani-Damaneh}, \binits{M.}}:
\batitle{Characterization of nonsmooth quasiconvex and pseudoconvex functions}.
\bjtitle{Journal of Mathematical Analysis and Applications}
\bvolume{330}(\bissue{2}),
\bfpage{1387}--\blpage{1392}
(\byear{2007})
\end{barticle}
\endbibitem

\bibitem[\protect\citeauthoryear{Kiwiel}{2001}]{kiwiel01}
\begin{barticle}
\bauthor{\bsnm{Kiwiel}, \binits{K.C.}}:
\batitle{Convergence and efficiency of subgradient methods for quasiconvex minimization}.
\bjtitle{Mathematical Programming}
\bvolume{90}(\bissue{1}),
\bfpage{1}--\blpage{25}
(\byear{2001})
\end{barticle}
\endbibitem

\bibitem[\protect\citeauthoryear{Ben-Tal and Nemirovski}{1997}]{bental97}
\begin{barticle}
\bauthor{\bsnm{Ben-Tal}, \binits{A.}},
\bauthor{\bsnm{Nemirovski}, \binits{A.}}:
\batitle{Robust truss topology design via semidefinite programming}.
\bjtitle{SIAM Journal on Optimization}
\bvolume{7}(\bissue{4}),
\bfpage{991}--\blpage{1016}
(\byear{1997})
\end{barticle}
\endbibitem

\bibitem[\protect\citeauthoryear{Kanno}{2015}]{kanno15}
\begin{barticle}
\bauthor{\bsnm{Kanno}, \binits{Y.}}:
\batitle{A note on formulations of robust compliance optimization under uncertain loads}.
\bjtitle{Journal of Structural and Construction Engineering}
\bvolume{80},
\bfpage{601}--\blpage{607}
(\byear{2015})
\end{barticle}
\endbibitem

\bibitem[\protect\citeauthoryear{Nishioka and Kanno}{2023}]{nishioka23orl}
\begin{botherref}
\oauthor{\bsnm{Nishioka}, \binits{A.}},
\oauthor{\bsnm{Kanno}, \binits{Y.}}:
A feasible smoothing accelerated projected gradient method for nonsmooth convex optimization.
arXiv preprint arXiv:2312.07050
(2023)
\end{botherref}
\endbibitem

\bibitem[\protect\citeauthoryear{Andreassen et~al.}{2011}]{andreassen11}
\begin{barticle}
\bauthor{\bsnm{Andreassen}, \binits{E.}},
\bauthor{\bsnm{Clausen}, \binits{A.}},
\bauthor{\bsnm{Schevenels}, \binits{M.}},
\bauthor{\bsnm{Lazarov}, \binits{B.S.}},
\bauthor{\bsnm{Sigmund}, \binits{O.}}:
\batitle{Efficient topology optimization in {MATLAB} using 88 lines of code}.
\bjtitle{Structural and Multidisciplinary Optimization}
\bvolume{43}(\bissue{1}),
\bfpage{1}--\blpage{16}
(\byear{2011})
\end{barticle}
\endbibitem

\bibitem[\protect\citeauthoryear{Takezawa et~al.}{2011}]{takezawa11}
\begin{barticle}
\bauthor{\bsnm{Takezawa}, \binits{A.}},
\bauthor{\bsnm{Nii}, \binits{S.}},
\bauthor{\bsnm{Kitamura}, \binits{M.}},
\bauthor{\bsnm{Kogiso}, \binits{N.}}:
\batitle{Topology optimization for worst load conditions based on the eigenvalue analysis of an aggregated linear system}.
\bjtitle{Computer Methods in Applied Mechanics and Engineering}
\bvolume{200},
\bfpage{2268}--\blpage{2281}
(\byear{2011})
\end{barticle}
\endbibitem

\bibitem[\protect\citeauthoryear{Cherkaev and Cherkaev}{2008}]{cherkaev08}
\begin{barticle}
\bauthor{\bsnm{Cherkaev}, \binits{E.}},
\bauthor{\bsnm{Cherkaev}, \binits{A.}}:
\batitle{Minimax optimization problem of structural design}.
\bjtitle{Computers \& Structures}
\bvolume{86}(\bissue{13--14}),
\bfpage{1426}--\blpage{1435}
(\byear{2008})
\end{barticle}
\endbibitem

\bibitem[\protect\citeauthoryear{Ding and Grimmer}{2023}]{ding23}
\begin{barticle}
\bauthor{\bsnm{Ding}, \binits{L.}},
\bauthor{\bsnm{Grimmer}, \binits{B.}}:
\batitle{Revisiting spectral bundle methods: Primal-dual (sub) linear convergence rates}.
\bjtitle{SIAM Journal on Optimization}
\bvolume{33}(\bissue{2}),
\bfpage{1305}--\blpage{1332}
(\byear{2023})
\end{barticle}
\endbibitem

\bibitem[\protect\citeauthoryear{Nishioka and Kanno}{2023}]{nishioka23smao}
\begin{barticle}
\bauthor{\bsnm{Nishioka}, \binits{A.}},
\bauthor{\bsnm{Kanno}, \binits{Y.}}:
\batitle{Smoothing inertial method for worst-case robust topology optimization under load uncertainty}.
\bjtitle{Structural and Multidisciplinary Optimization}
\bvolume{66},
\bfpage{82}
(\byear{2023})
\end{barticle}
\endbibitem

\bibitem[\protect\citeauthoryear{Beck}{2017}]{beck17}
\begin{bbook}
\bauthor{\bsnm{Beck}, \binits{A.}}:
\bbtitle{First-order Methods in Optimization}.
\bpublisher{SIAM},
\blocation{Philadelphia}
(\byear{2017})
\end{bbook}
\endbibitem

\end{thebibliography}

\end{document}